\pdfoutput=1
\documentclass[11pt]{amsart}

\usepackage{booktabs}

\usepackage{amsfonts}
\usepackage{amsmath}
\usepackage{amssymb}
\usepackage{hyperref}
\usepackage{todonotes}
\usepackage{algorithm}
\usepackage[noend]{algorithmic}
\usepackage{enumitem}
\usepackage{comment}%
\usepackage{graphicx}
\usepackage{microtype}
\usepackage[english]{babel} 
\usepackage{tikz}\usetikzlibrary{automata,calc}
\usepackage{epstopdf}
\newlength{\lw}
\newlength{\st}
\newlength{\qs}
\newlength{\nd}
\tikzstyle{every picture}+=[auto]
\tikzstyle{every picture}+=[bend angle=10]
\tikzstyle{every picture}+=[join=round]
\tikzstyle{every picture}+=[cap=butt]
\tikzstyle{every picture}+=[line width=\lw]
\tikzstyle{every picture}+=[double distance=2\lw]
\tikzstyle{every picture}+=[shorten >=\st]
\tikzstyle{every picture}+=[node distance=\nd]
\tikzstyle{every loop}=[->,shorten >=\st]
\tikzstyle{place}=[circle,draw,minimum size=\qs]
\tikzstyle{transition}=[rectangle,draw,minimum size=\qs]
\tikzstyle{invisible}=[draw=none,inner sep=0pt,minimum height=0pt]
\newenvironment{petrinet}[1][]{\begin{center}\begin{tikzpicture}}{\end{tikzpicture}\end{center}}

\subjclass[2010]{68W10(primary), and 68W30, 14B05, 14Q99(secondary)} 

\newtheorem{theorem}{Theorem}[section]
\theoremstyle{definition}
\newtheorem{df}[theorem]{Definition}

\newtheorem{notation}[theorem]{Notation}
\theoremstyle{definition}
\newtheorem{Example}[theorem]{Example}
\theoremstyle{definition}
\newtheorem{rem}[theorem]{Remark}

\newtheorem{lem}[theorem]{Lemma}
\newtheorem{construction}[theorem]{Construction}

\newcommand{\parens}[3]{#1#3#2}
\newcommand{\K}[1]{\parens{\left(}{\right)}{#1}}
\newcommand{\card}[1]{\parens{\left|}{\right|}{#1}}
\newcommand{\set}[1]{\parens{\left\{}{\right\}}{#1}}
\newcommand{\define}[1]{\emph{#1}}
\newcommand{\NN}{\mathbb{N}}

\DeclareMathOperator{\Sing}{Sing}
\DeclareMathOperator{\codim}{codim}
\DeclareMathOperator{\lex}{lex}
\DeclareMathOperator{\ord}{ord}
\DeclareMathOperator{\Spec}{Spec}

\newcommand{\bk}{{\Bbbk}}
\newcommand{\bK}{{\mathbb K}}

\usepackage{enumitem}
\newlist{cvdesc}{description}{1}
\setlist[cvdesc]{nosep,
labelindent=0pt,
labelwidth=2.8cm,
labelsep*=0.2cm,
leftmargin=3cm,
font=\normalfont,
align=right}
\newlist{compactitem}{itemize}{3}
\setlist[compactitem]{nosep,itemsep=2pt,topsep=3pt,labelindent=1em,leftmargin=2em}
\setlist[compactitem,1]{label=\textbullet}
\setlist[compactitem,2]{label=--}

\begin{document}
\title[Massively parallel computations in algebraic geometry]{
Towards Massively Parallel Computations in Algebraic Geometry
}
\author{Janko~B\"ohm}
\address{Janko~B\"ohm, Department of Mathematics, University of Kaiserslautern,  Erwin-Schr\"odinger-Str., 67663 Kaiserslautern, Germany}
\email{boehm@mathematik.uni-kl.de}
\author{Wolfram~Decker}
\address{Wolfram~Decker, Department of Mathematics, University of Kaiserslautern,  Erwin-Schr\"odinger-Str., 67663 Kaiserslautern, Germany}
\email{decker@mathematik.uni-kl.de}
\author{Anne~Fr\"uhbis-Kr\"uger}
\address{Anne~Fr\"uhbis-Kr\"uger, Institut f\"ur algebraische Geometrie, Leibniz Universit\"at Hannover,  Welfengarten 1, 30167 Hannover, Germany}
\email{anne@math.uni-hannover.de}

\author{Franz-Josef Pfreundt}
\address{Franz-Josef Pfreundt, Competence Center High Performance Computing, Fraunhofer ITWM,  Fraunhofer-Platz 1, 67663 Kaiserslautern, Germany}
\email{franz-josef.pfreundt@itwm.fhg.de}

\author{Mirko Rahn}
\address{Mirko Rahn, Competence Center High Performance Computing, Fraunhofer ITWM,  Fraunhofer-Platz 1, 67663 Kaiserslautern, Germany}
\email{mirko.rahn@itwm.fhg.de}

\author{Lukas~Ristau}
\address{Lukas~Ristau, Department of Mathematics, University of Kaiserslautern,  Erwin-Schr\"odinger-Str., and  Competence Center High Performance Computing, Fraunhofer ITWM,  Fraunhofer-Platz 1, 67663 Kaiserslautern, Germany, }
\email{ristau@mathematik.uni-kl.de}

\renewcommand\shortauthors{B\"ohm, J. et al}

\begin{abstract}
Introducing parallelism and exploring its use is still a fundamental challenge for the computer algebra community.
In high performance numerical simulation, on the other hand, transparent environments for distributed 
computing which follow the principle of separating coordination  and computation have been a success story 
for many years. In this paper, we explore the potential of using  this principle in the context of computer 
algebra. More precisely, we combine two well-established systems: The mathematics we are interested
in is implemented in the computer algebra system {\textsc{Singular}}, whose focus is on polynomial 
computations, while the coordination is left to the workflow management system GPI-Space, which relies 
on Petri nets as its mathematical modeling language, and has been successfully used for coordinating 
the parallel execution (autoparallelization) of academic codes as well as for commercial software in 
application areas such as seismic data processing. The result of our efforts is a major step towards a
framework for massively parallel computations in the application areas of {\textsc{Singular}},
specifically in commutative algebra and algebraic geometry. As a first test case for this framework, we have 
modeled and implemented a hybrid smoothness test for algebraic varieties which combines ideas from
Hironaka's celebrated desingularization proof with the classical Jacobian criterion. Applying our implementation 
to two examples originating from current research in algebraic geometry, one of which cannot be handled 
by other means, we illustrate the behavior of the smoothness test within our framework,
and investigate how the computations scale up to 256 cores.

\end{abstract}

\keywords{Computer algebra, Singular, distributed computing, GPI-Space, Petri nets,
computational algebraic geometry, Hironaka desingularization, smoothness test,
surfaces of general type}

\thanks{This work has been supported by the German Research Foundation (DFG) through SPP 1489 and TRR 195, Project II.5.}

\maketitle

\section{Introduction}

Experiments based on calculating examples have always played a key role in mathematical research. 
Advanced hardware structures paired with sophisticated mathematical software tools allow for far reaching experiments 
which were previously unimaginable. In the realm of algebra and its applications, where exact calculations are 
inevitable, the desired software tools are provided by computer algebra systems. In order to take full advantage of
modern multicore computers and high-performance clusters, the
computer algebra community must provide parallelism in their systems.  
This will boost the performance of the systems to a new level, thus extending the scope of applications significantly.
However, while there has been a lot of progress in this direction in numerical computing, achieving parallelization 
in symbolic computing is still a tremendous challenge both from a mathematical and technical point of view.

On the mathematical side, there are some algorithms whose basic strategy is inherently parallel, whereas many
others are sequential in nature. The systematic design and implementation of parallel algorithms is a major 
task for the years to come. On the technical side, models for parallel computing have long been studied in 
computer science. These differ in several fundamental aspects. Roughly, two basic paradigms can be distinguished 
according to assumptions on the underlying hardware. The shared memory based models allow several different 
computational processes (called threads) to access the same data in memory, while the distributed models run 
many independent processes which need to communicate their progress to one or several of the other processes.
Creating  the prerequisites for writing parallel code in a computer algebra system originally designed for sequential 
processes requires considerable  efforts which affect all levels of the system. 

In this paper, we explore an alternative way of introducing parallelism into computer algebra
computations. This approach is  non-intrusive and allows for distributed computing. It is based 
on the principle of separating coordination and computation, a principle which has already been pursued with great success 
in high performance numerical simulation. Specifically, we rely on the workflow management system GPI-Space 
\cite{GPI} for coordination, while the mathematics we are interested in is implemented in the computer algebra 
system {\sc Singular} \cite{Singular}.

 {\sc Singular} is under development at TU Kaiserslautern, focuses on polynomial computations, and has been 
successfully used in application areas such as algebraic geometry and singularity theory. GPI-Space, on the 
other hand, is under development at Fraunhofer ITWM Kaiserslautern, and has been successfully used for 
coordinating the parallel execution (autoparallelization) of academic codes as well as for commercial 
software in application areas such as seismic data processing.
As its mathematical modeling language,  GPI-Space relies on Petri nets, which are specifically designed 
to model concurrent systems, and yield both data parallelism and task parallelism. In fact, GPI-Space 
is not only  able to automatically balance, to automatically scale up to huge machines, or to tolerate machine 
failures, but can also use existing legacy applications and integrate them, without requiring any change to them.
In our case, {\sc Singular} calls GPI-Space, which, in turn,  manages several (many) instances of 
{\sc Singular} in its existing binary form (without any need for changes). The experiments carried
through so far are promising and indicate that we are on our way towards a convenient framework 
for massively parallel computations in {\sc Singular}.

One of the central tasks of computational algebraic geometry is the explicit construction of objects 
with prescribed properties, for instance to find counterexamples to conjectures or to construct general 
members of moduli spaces. Arguably, the most important property to be checked here is smoothness.
Classically, this means to apply the Jacobian criterion: If $X\subset \mathbb A^n_\bK$ 
(respectively $X\subset \mathbb P^n_\bK$) is an equidimensional affine (respectively projective) 
algebraic variety of dimension $d$ with defining equations $f_1=\dots =f_s=0$, compute a 
Gr\"obner basis of the ideal generated by the $f_i$ together with the $(n-d)\times (n-d)$ minors of the 
Jacobian matrix of the $f_i$ in order to check whether this ideal defines the empty set.
The resulting process is predominantly sequential. It is typically expensive (if not unfeasible), 
especially in cases where the codimension $n-d$ is large.

In \cite{smoothtst},  an alternative smoothness test has been suggested by the first and third author 
(see \cite{smoothtstlib} for the implementation in \textsc{Singular}). This test builds on ideas from 
Hironaka's celebrated desingularization proof \cite{Hir} and is 
intrinsically parallel. To explore the potential of our framework, we have modeled and implemented an enhanced 
version of the test  which is interesting in its own right. Following \cite{smoothtst}, we take our cue from the fact
that each smooth variety is locally a complete intersection. Roughly, the idea is then to apply Hironaka's method  
of descending induction by hypersurfaces of maximal contact (in its constructive version by Bravo, Encinas, and 
Villamayor \cite{BEV}). This allows us either to detect non-smoothness during  the process, or to finally realize
a finite covering of $X$ by affine charts such that in each chart, $X$ is  given as a smooth complete intersection. 
More precisely, at each iteration step,  our algorithm starts from finitely many affine charts $U_i$ which cover $X$,
together with varieties $W_i$ and embeddings $X\cap U_i\subset W_i\cap U_i$ such that each $W_i\cap U_i$ is a 
smooth complete intersection in $U_i$.  Providing a constructive version of Hironaka's termination criterion, 
the algorithm then either detects that $X$ is singular in one $U_i$, and terminates, or constructs for each 
$i$ finitely many affine charts $U'_{ij}$ which cover $X\cap U_i$, together with varieties $W'_{ij}\subset W_i$ and embeddings 
$X\cap U'_{ij}\subset W'_{ij}\cap U'_{ij}$ such that each $W'_{ij}\cap U'_{ij}$  is a smooth complete intersection in $U'_{ij}$
whose codimension is one less than that of $X\cap U_i$ in $W_i\cap U_i$. Since at each step, the computations in one 
chart do not depend on results from the other charts, the algorithm is indeed parallel in nature. Moreover, since our 
implementation branches into all available choices of charts in a massively parallel way, and terminates once $X$ is 
completely covered by charts, it will automatically determine a choice of charts which leads to the smoothness 
certificate in the fastest possible way. 

In fact, there is one more twist: As experiments show, see \cite{smoothtst}, the smoothness test is most effective
in a hybrid version which makes use of the above ideas to reduce the general problem to checking smoothness 
in finitely many embedded situations $X\cap U\subset W\cap U$ of low codimension, and applies (a relative 
version of) the Jacobian criterion there.

Our paper is organized as follows. In Section \ref{sec:smoothness}, we briefly review smoothness 
and recall the Jacobian criterion. In Section \ref{sec hybrid smoothness test}, we summarize what we need
from Hironaka-style desingularization and develop our smoothness test. Section \ref{sec GPI} contains a discussion
of GPI-Space and Petri nets which prepares for Section \ref{sec Petri smooth},  where we show how to model 
our test in terms of Petri nets. This forms the basis for the implementation of the test using 
{\sc{Singular}} within GPI-space. Finally, in Section \ref{sec timings}, we illustrate the behavior 
of the smoothness test and its implementation by checking two examples from current research 
in algebraic geometry. These examples are surfaces of general type, one of which cannot be handled by other means.

\section{Smoothness and the Jacobian Criterion}\label{sec:smoothness}

We describe the geometry behind our algorithm in the classical language of algebraic varieties
over an algebraically closed field. Because smoothness is a local property, and each 
quasiprojective (algebraic) variety admits an open affine covering, we restrict 
our attention to  affine (algebraic) varieties. 

Let $\bK$ be an algebraically closed field. Write $\mathbb A^n_{\bK}$
for the affine $n$-space over $\bK$. An affine variety (over $\bK$) is the common vanishing locus 
$V(f_1, \dots, f_r)\subset \mathbb A^n_{\bK}$ of finitely many polynomials $f_i\in\bK[x_1,\dots, x_n]$.
If $Z$ is such a variety, let 
$$
I_Z=\{f\in \bK[x_1,\ldots,x_n] \mid f(p)=0 \text{ for all } p\in Z\}\subset \bK[x_1,\ldots,x_n]
$$ 
be its vanishing ideal, let $\bK[Z]=\bK[x_1,\ldots,x_n]/I_Z$ be its ring of polynomial functions,
and let $\dim Z = \dim \bK[Z]$ be its dimension.

Given a polynomial $h\in\bK[x_1,\ldots,x_n]$, we write $$D(h) = \mathbb
A^n_{\bK}\setminus V(h)=\{p\in \mathbb A^n_{\bK} \mid h(p) \neq 0\}$$
for the principal open set  defined by $h$, and $\mathcal O_{Z}(Z\cap D(h))$ 
for the ring of regular functions on $Z\cap D(h)$. If $p\in Z$ is a point,
we write $\mathcal{O}_{Z,p}$ for the local ring
of $Z$ at $p$, and $\mathfrak{m}_{Z,p}$ for the maximal ideal of $\mathcal{O}_{Z,p}$. 
Recall that both rings $\mathcal O_{Z}(Z\cap D(h))$ and $\mathcal{O}_{Z,p}$ are localizations of $\bK[Z]$:
Allow powers of the polynomial function defined by $h$ on $Z$ and polynomial functions
on $Z$ not vanishing at $p$ as denominators, respectively. 

Relying on the trick of Rabinowitch, we regard $Z\cap D(h)$ as an affine 
variety: If $I_Z=\langle f_1,\dots, f_s\rangle$, identify $Z\cap D(h)$ with the vanishing 
locus $$V(f_1,\dots, f_s, ht-1)\subset \mathbb A^{n+1}_{\bK},$$ where $t$ is an extra 
variable.

The tangent space at a point $p=(a_1,\dots, a_n)\in Z$ is the linear variety
$$
T_p Z = V(d_p(f) \mid f\in I_Z)\subset \mathbb A^n_{\bK},
$$
where $d_pf$ is the differential of $f$ at $p$:
$$
d_p f=\sum_{i=1}^n \frac{\partial f}{\partial x_i}(p)(x_i-a_i)
\in \bK[x_1,\dots, x_n].
$$
We have
$$
\dim T_p Z \geq \max \{\dim V \mid V \text{ is an irreducible component of } Z \text{ through } p\},
$$
and say that $Z$ is \emph{smooth at $p$} if these numbers are equal. Equivalently, $\mathcal{O}_{Z,p}$ is a regular local ring. 
Otherwise, $Z$ is \emph{singular at $p$}. 
The variety $Z$ is \emph{smooth}  if it is smooth at each of its points. 

Recall that a variety $Z$ is equidimensional if all its irreducible components
have the same dimension. Algebraically this means that the ideal $I_Z$ is equidimensional, 
that is, all associated primes of $I_Z$ have the same dimension.

\begin{theorem}[Jacobian Criterion]
Let $\bK$ be an algebraically closed field, and let $Z = V(f_1, \dots, f_s)\subset \mathbb A^n_{\bK}$ 
be an affine variety which is equidimensional of dimension $d$.
Write $I_{n-d}\left({\mathcal J}\right)$ for the the ideal generated by the
$(n-d) \times (n-d)$ minors of the  Jacobian matrix
${\mathcal J} = \left({\partial f_i \over \partial x_j}\right)$.
If  $I_{n-d}\left({\mathcal J}\right)+I_Z = \langle 1\rangle$, then $Z$ is smooth,
and the ideal $\langle f_1, \dots, f_s\rangle \subset \bK[x_1,\ldots,x_n]$
is equal to the vanishing ideal $I_Z$ of $Z$.  In particular, $\langle f_1, \dots, f_s\rangle$ 
is a radical ideal.
\end{theorem}

If\ $Y\subset Z\subset \mathbb A^n_{\bK}$ are two affine varieties, the vanishing ideal
$I_{Y,Z}$ of $Y$ in $Z$ is the ideal generated by $I_Y$ in $\bK[Z]$. If $Z$ is equidimensional,
we write $\codim_Z Y=\dim Z - \dim Y$ for the codimension of $Y$ in $Z$,
and say that $Y$ is a \emph{complete intersection} in $Z$ if 
$I_{Y,Z}$ can be generated by $\codim_Z Y=\codim I_{Y,Z}$ elements (then  $Y$
and  $I_{Y,Z}$ are equidimensional as well).

\section{A Hybrid smoothness test}\label{sec hybrid smoothness test}

In this section, we present the details of our hybrid smoothness test which, as already
outlined in the introduction, combines the Jacobian criterion with ideas from Hironaka's
landmark paper on the resolution of singularities \cite{Hir} in which Hironaka proved 
that such resolutions exist, provided we work in characteristic zero.

For detecting non-smoothness and controlling the resolution process, Hironaka developed a theory of standard 
bases for local rings and their completions (see \cite[Chapter 1]{GP} for the algorithmic aspects of 
standard bases). Based on this, he defined several invariants controlling the desingularization process.
The  so-called $\nu^{*}$-invariant generalizes the order of a power series. As some sort
of motivation, we recall its definition in the analytic setting: Let $(X,0) \subset({\mathbb{A}}_{\bK}^{n},0)$ 
be an analytic space germ over an algebraically closed field $\bK$ of characteristic zero, let 
$\bK\{x_{1},\ldots,x_{n}\}$ be the ring of convergent power series with coefficients in $\bK$, and
let $I_{X,0} \subset \bK\{x_{1},\ldots,x_{n}\}$ be the defining ideal of $(X,0)$. If
$f_{1},\dots,f_{s}$ form a minimal standard basis of $I_{X,0}$, and the $f_i$ are 
sorted by increasing order $\ord(f_i)$, then set
$$
\nu^{*}(X,0) = (\ord(f_1), \dots, \ord(f_s)).
$$ 
This invariant is the key to Hironaka's termination criterion: 
The germ $(X,0)$ is singular iff at least one of the entries of $\nu^{*}(X,0)$ is $>1$.

In the algebraic setting of this paper, let $X\subset{\mathbb{A}}_{\bK}^{n}$ be an 
equidimensional affine variety, with vanishing ideal $I_X\subset \bK[x_1,\ldots,x_n]$, 
where $\bK$ is an algebraically closed field  of arbitrary characteristic. Working in 
arbitrary characteristic allows for a broader range of potential applications, and is not a problem
since we will only rely on results from Hironaka's papers which also hold in positive characteristic. 

To formulate Hironaka's criterion in the algebraic setting,  we first recall how to extend the notion of order:

\begin{df}
If $(R, \mathfrak{m})$ is any local Noetherian ring, and  $0\neq f\in R$ is
any element, then the \emph{order} of $f$ is defined by setting 
$$
\ord(f)=\max\{k\in \mathbb{N}\mid f\in \mathfrak{m}^k\}.
$$
\end{df}

\begin{df}[\cite{Hir, Hir1967}]  
\label{def:nu-star}
With notation as above, let $p\in X$. If $f_{1},\dots,f_{s}$ form a minimal standard basis of the extended ideal 
$I_X {\mathcal O}_{{\mathbb A}^n_{\bK},p}$ with respect to a local degree ordering, and the $f_i$ are sorted 
by increasing order, set
$$
\nu^{*}(X,p) = (\ord(f_1), \dots, \ord(f_s)).
$$ 
\end{df}
\begin{lem}[\cite{Hir, Hir1967}]
The sequence $\nu^{*}(X,p)$ depends only on $X$ and $p$.
\end{lem}

\begin{rem}
\label{rem:comp-nu-star}
Note that $\nu^{*}(X,p)$ can be determined algorithmically:  A minimal standard basis as required  is obtained 
by translating $p$ to the origin and applying Mora's tangent cone algorithm (see \cite{TangentCone}, \cite{GP}). 
\end{rem}

Hironaka's criterion can now be stated as follows: 

\begin{lem}[\cite{Hir}, Chapter III] 
\label{crit-Hir}
The variety $X$ is singular at $p\in X$ iff 
\begin{equation}
\nu^{*}(X,p) >_{\lex} (1,\dots,1)\in\mathbb N^{\operatorname{codim} X},
\end{equation}
where $>_{\lex}$ denotes the lexicographical ordering.
\end{lem}

Note that if $X$ is singular at $p$, then the length of $\nu^{*}(X,p)$ may be larger than $\operatorname{codim}(X)$,
but at least one of the first ${\operatorname{codim}(X)}$ entries will be $>1$. 

Hironaka's criterion is not of immediate practical use for us: We cannot examine each single point $p\in X$.
Fortunately, solutions to this problem have been suggested by various authors while establishing
constructive versions of Hironaka's resolution process (see, for example, \cite{BM}, \cite{BEV}, \cite{V1}). Here,
we follow the approach of Bravo, Encinas, and Villamayor \cite{BEV} which is best-suited
for our purposes. Their simplified proof of desingularization replaces
 local standard bases at individual points by the use of loci of maximal order.  These loci are obtained by polynomial computations  
in finitely many charts  (see \cite[Section 4.2]{FK1}). Loci of maximal order can be used to  find so-called hypersurfaces of
maximal contact, which again only exist locally in charts. In a Hironaka style resolution process, hypersurfaces
of maximal contact allow for a descending induction on the dimension of the respective ambient space. 
That such hypersurfaces generally do not exist in positive characteristic is
a key obstacle for extending Hironaka's ideas to positive characteristic.

In our context, we encounter a particularly simple special case of all this.
We suppose that we are given an embedding $X\subset W$, where $W$ is
a smooth complete intersection in ${\mathbb{A}}_{\bK}^{n}$, say
of codimension $r$. In particular, $W$ is equidimensional of dimension 
$$d=n-r.$$ 
The idea is then to first check whether the locus of order at least two is 
non-empty. In this case, $X$ is singular. Otherwise, we can find a finite 
covering of $X$ by affine charts and in each chart a hypersurface of maximal 
contact whose construction relies only on the suitable choice of one of the generators 
of $I_X$ together with one first order partial derivative of this generator\footnote{As a result, the difficulties of resolution of singularities in positive characteristic do 
not occur in our setting, see Lemma \ref{lem max contact} below.}. In each chart,
we then consider the hypersurface of maximal contact as the new ambient space 
of $X$, and proceed by iteration.

The resulting process allows us to decide at each step of the iteration whether
there is a point $p\in X$ such that the next entry of $\nu^{*}(X,p)$ is $\geq 2$.
To give a more precise statement, we suppose that $X$ has positive codimension in $W$
(otherwise, $X$ is necessarily smooth).
Crucial for obtaining information on an individual entry of $\nu^{\ast}$ is the order of ideals:

\begin{df}
If $(R, \mathfrak{m})$ is any local Noetherian ring, and $\langle 0 \rangle\neq J=\langle 
h_1,\dots,h_t\rangle\subset R$ is any ideal, then the \emph{order} of $J$ is defined by setting
$$
\ord(J)=\max\{k\in \mathbb{N}\mid J\subset \mathfrak{m}^k\}
=\operatorname{min}\left \{\operatorname{ord}(h_i) \mid i=1,\dots, t\right\}.
$$
\end{df}

\noindent
In our geometric setup, we apply this as follows:
Given an ideal $\langle 0 \rangle \neq I\subset \bK[W]$ and a point $p\in W$, 
the \emph{order $\ord_p(I)$ of $I$ \emph{at} $p$} is defined to be the order of the extended ideal
$I \mathcal O_{W,p}$. For $0\neq f\in \bK[W]$ we similarly define $\ord_p(f)$ as the order of 
the image of $f$ in $\mathcal O_{W,p}$.

\begin{df}
\label{def:order-gen}
With notation as above, 
for any integer $b \in{\mathbb{N}}$, the \define{locus of order at least $b$} of the vanishing ideal $I_{X,W}$ is 
$$
\Sing(I_{X,W},b)=\left\{  p\in X\mid\operatorname{ord}_{p}(I_{X,W})\geq b\right\}.
$$ 
\end{df}

\begin{rem}[\cite{Hir}, Chapter III] Note that the loci $\Sing(I_{X,W},b)$ are Zariski closed
since the function 
$$X\rightarrow \mathbb N, \ p \mapsto \operatorname{ord}_{p}(I_{X,W}),$$
is Zariski upper semi-continuous.
\end{rem}

\begin{rem}
\label{rem:criterion-singular-I}
With notation as above, let a point $p\in X$ be given. Then the first $r$ elements
of a minimal standard basis of $I_X {\mathcal O}_{{\mathbb A}^n_\bK,p}$ as in Definition 
\ref{def:nu-star} must have order 1 by our assumptions on $W$, that is, the first $r$ entries of $\nu^*(X,p)$ are equal to~$1$. 
On the other hand, if $\operatorname{ord}_{p}(I_{X,W})\geq 2$, then the $(r+1)$-st entry of 
$\nu^*(X,p)$  is $\geq 2$. Hence, in this case, $X$ is singular at $p$
since the codimension of $X$ in $\mathbb{A}_{\bK}^{n}$ is  at least $r+1$ by our assumptions. 
\end{rem}

In terms of loci of order at least two this amounts to:

\begin{lem}
\label{lem: criterion-singular-II}
With notation as above, $X$ is singular if   
$$\Sing(I_{X,W},2)\not=\emptyset.$$
\end{lem}
\begin{proof}
Clear from Remark \ref{rem:criterion-singular-I}.
\end{proof}

To determine the loci   $\Sing(I_{X,W},b)$ in a Zariski neighbourhood
of a point~$p\in X$ explicitly, derivatives with respect 
to a regular system of parameters of $W$ at $p$ are the method of choice:
See \cite[p. 404]{BEV} for characteristic zero, and \cite[Sections 2.5 and 2.6]{GiraudPos}
for positive characteristic using Hasse derivatives.
For a more detailed description, fix a point $p\in W$. According to our assumptions, the local 
ring $\mathcal{O}_{W,p}$ is regular of dimension $d$. So we can find 
a regular system of parameters $X_{p,1},\dots, X_{p,d}$ for $\mathcal{O}_{W,p}$. That is, 
$X_{p,1},\dots, X_{p,d}$  form a minimal set of generators for $\mathfrak{m}_{W,p}$.
By the Cohen structure theorem, we may, thus, think of the completion 
$\widehat{\mathcal O_{W,p}}$ as a formal power series ring in $d$ variables
(see \cite[Proposition 10.16]{Eis}): The map
$$
\Phi:\bK[[y_1, \dots, y_d]] \rightarrow \widehat{\mathcal O_{W,p}}, \ y_i\mapsto X_{p,i},
$$
is an isomorphism of local rings. In particular, 
the order of an element $f\in \bK[W]$ at $p$ coincides with the order of the
formal power series $\Phi^{-1}(f)\in \bK[[y_1, \dots, y_d]]$. The latter, in turn, 
can be computed as follows:

\begin{lem}[\cite{BEV}, \cite{GiraudPos}] 
Let $R=\bK[[y_1,\dots,y_d]]$, let ${\mathfrak m}=\langle y_1,\dots,y_d \rangle$
be the maximal ideal of $R$, and  let $F\in R \setminus \{0\}$. Then 
$$\operatorname{ord}(F)= \operatorname{min}\left \{m \in {\mathbb N} \bigm| 
\frac{\partial^{a} F}{\partial y^{a}}\not\in {\mathfrak m}   \textrm{ for some } a \in {\mathbb N}^n
 \textrm{ with } |a|=m \right \}, 
$$ 
where the derivatives denote the usual formal derivatives in characteristic zero, and 
Hasse derivatives in positive characteristic.
\end{lem}

As we focus on the locus $\Sing(I_{X,W},b)$ with $b=2$, only first order formal derivatives play 
a role for us. Since these  derivatives coincide with the first order Hasse derivatives, we do not need to discuss 
Hasse derivatives here.

\begin{df}\label{def deriv}
In the situation above, we use the isomorphism $\Phi$ of the Cohen structure theorem to define 
\emph{first order derivatives} of elements $f\in\widehat{\mathcal O_{W,p}}$ 
\emph{with respect to the regular system of parameters} $X_{p,1},\dots, X_{p,d}$:
Set
$$
\frac{\partial f}{\partial X_{p,j}}=\Phi\hspace{-1mm}\left(\frac{\partial \Phi^{-1}(f)}{\partial y_{j}}\right)\in\widehat{\mathcal O_{W,p}}, \text{ for } j=1,\dots, d.
$$ 
\end{df}

We summarize our discussion so far. If $I_{X,W}$ is given by a set of generators
$f_{r+1}, \dots, f_s\in\bK[W]\setminus \{0\}$, and if $p\in X$, then $p\in\Sing(I_{X,W},2)$  iff $\operatorname{ord}_p(f_j)>1$
for all $j$. In this case, $X$ is singular at $p$. Furthermore, if $0\neq f\in \bK[W]$ is any element, $p\in W$ is any point,
and $X_{p,1},\dots, X_{p,d}$ is a regular system of parameters for $\mathcal{O}_{W,p}$, then $\operatorname{ord}_p(f)>1$ iff  
\begin{equation}
\label{not:Delta-local}
1 \not\in \Delta_p(f) :=\left\langle f, \frac{\partial f}{\partial X_{p,1}},\dots,
\frac{\partial f}{\partial X_{p,d}} \right\rangle_{\widehat{\mathcal{O}_{W,p}}}\subset \widehat{\mathcal{O}_{W,p}}.
\end{equation}
Now, as before, we cannot examine each point individually.  The following arguments will
allow us to remedy this situation in Lemma \ref{locus order 2} below. We begin by showing 
that there is a locally consistent way of choosing regular systems of parameters:

\begin{lem}\label{lem covering by minors} 
Let $I_W=\langle f_1,\dots,f_r \rangle \subset \bK[x_1,\dots,x_n]$,
and let ${\mathcal J} = \left({\partial f_i \over \partial x_j}\right)$ be the Jacobian matrix of 
$f_1,\dots,f_r$. Then there is a finite covering of $W$ by principal 
open subsets $D(h)\subset \mathbb{A}_\bK^n$ such that:
\begin{enumerate}
\item Each polynomial $h$ is a maximal minor of ${\mathcal J}$.
\item For each $h$, the variables $x_j$ not used for differentiation in forming
the minor $h$ induce by translation a regular system of parameters for every local ring  
${\mathcal O_{W,p}}$, $p\in W\cap D(h)$.
\end{enumerate}
For each $h$, we refer to such a choice of a local system of parameters at all points of $W\cap D(h)$ as a \emph{consistent choice}.
\end{lem}

\begin{proof}
Consider a point $p_0\in W$. Then, by the  Jacobian criterion,
there is at least one minor $h=\det(M)$ of ${\mathcal J}$ of size $r$ such that
$h(p_0) \neq 0$ (recall that we assume that $W$ is smooth). Suppose for simplicity that $h$ involves the last $r$ columns of
${\mathcal J}$, and let $p=(a_1,\dots,a_n)$ be any point of $W\cap D(h)$. Then the  images of 
$x_{1}-a_{1}, \ldots,x_d-a_d, f_1,\ldots,f_r$  in ${\mathcal O _{{\mathbb A}^n_\bK,p}}$ are actually contained in 
${\mathfrak m}_{{\mathbb A}^n_\bK,p}$ and represent a $\bK$-basis of the Zariski tangent space
${\mathfrak m}_{{\mathbb A}^n_\bK,p}/{\mathfrak m}_{{\mathbb A}^n_\bK,p}^2$. 
Hence, by Nakayama's lemma, they form a minimal set of generators for ${\mathfrak m}_{{\mathbb A}^n_\bK,p}$. 
Since $f_1,\dots,f_r$ are mapped to zero when we pass to ${\mathcal O}_{W,p}$, the images of  
$x_{1}-a_{1}, \ldots,x_d-a_d$ in ${\mathcal O}_{W,p}$ form a regular system of parameters for 
${\mathcal O}_{W,p}$. The result follows because $W$ is quasi-compact in the Zariski topology.
\end{proof}

\begin{notation}
\label{notation:pos}
For further considerations, we retain the notation of the lemma and its proof. Fix one principal open subset
$D(h) \subset {\mathbb A}^n_\bK$ as in the lemma. 
Suppose that $h=\det(M)$ involves the last $r$ columns of  the Jacobian matrix ${\mathcal J}$. 
Furthermore, fix one element $0\neq f\in\bK[W]$.
\end{notation}
\noindent
We now show how to find an ideal $\Delta(f) \subset {\mathcal O}_{W}(W\cap D(h))$ such that
$$\Delta(f)\;\! \widehat{\mathcal{O}_{W,p}}=\Delta_p(f) \;\text{ for each point }\; p\in W\cap D(h),$$
where $\Delta_p(f)$ is defined as in \eqref{not:Delta-local}. Technically, we manipulate polynomials, 
starting from a polynomial in $\bK[x_1,\dots,x_n]$ representing $f$. By abuse of notation, we denote 
this polynomial again by $f$. 
\begin{construction}\label{con deriv}
We construct a polynomial $\tilde{f}\in \bK[x_1,\dots, x_n]$ whose image 
in ${\mathcal O}_{W}(W\cap D(h))$ coincides with that of $f$, and whose partial derivatives 
$\frac{\partial \tilde{f}}{\partial x_i}$, $i=d+1,\dots, n$, are mapped to zero in 
${\mathcal O}_{W}(W\cap D(h))/\langle f \rangle$.
For this, let $A$ be the matrix of cofactors of $M$. Then
$$A \cdot M = h \cdot E_{r},$$
where $E_r$ is the $r\times r$ identity matrix. Moreover, if $I\subset\bK[x_1,\dots,x_n]$ is the ideal
generated by the entries of  the vector $(\tilde{f_1},\dots,\tilde{f}_{r})^T =A \cdot (f_1,\dots,f_{r})^T$, then the extended ideals  
$I\mathcal{O}_{{\mathbb A}^n_\bK}(D(h))$ and $I_W\mathcal{O}_{{\mathbb A}^n_\bK}(D(h))$ coincide
since $h$ is a unit in $\mathcal{O}_{{\mathbb A}^n_\bK}(D(h))$.

Let $\widetilde{\mathcal J}=\left({\partial {\tilde{f_i}} \over \partial {x_j}}\right)$ be the Jacobian matrix of 
$\tilde{f_1},\dots,\tilde{f}_{r}$. Then the matrix $\widetilde{{\mathcal J}}\vert_{W\cap D(h)}$ obtained 
by mapping the entries of $\widetilde{{\mathcal J}}$ to $\mathcal{O}_W (W\cap D(h))$ is of type
$$\widetilde{{\mathcal J}}\vert_{W\cap D(h)} = \begin{pmatrix} \hspace{2mm}* & \mid &  h \cdot E_{r}\end{pmatrix}$$
(apply the product rule).
In $\mathcal{O}_{{\mathbb A}^n_\bK}(D(h))$, the polynomial $\hat{f}=h\cdot f$ represents the same class as  
$f$. Moreover, modulo $f$, each partial 
derivative of $\hat{f}$  is divisible by $h$. Hence, after suitable row operations, the partial derivatives in the right hand lower block 
of the Jacobian matrix of $\tilde{f}_1,\dots, \tilde{f}_{r}, \hat{f}$ are mapped to zero in ${\mathcal O}_W (W\cap D(h))/\langle f\rangle$:

\[\scalebox{0.8}{
$\left(
\begin{tabular}
[c]{ccc|ccc}%
&&& $h$ &   &$0$   \\
& $\ast$ && & $\ddots$ \\
& &&$0$&  & $h$ \\\hline \rule{0pt}{1.2\normalbaselineskip}
$\frac{\partial \hat{f}}{\partial x_1}$& $\dots$ &
          $ \frac{\partial \hat{f}}{\partial x_{d}}$ & $ \frac{\partial \hat{f}}{\partial x_{d+1}}$& $\dots$ &
          $ \frac{\partial \hat{f}}{\partial x_{n}}$%
\end{tabular}
\right)  \longmapsto\left(
\begin{tabular}
[c]{ccc|ccc}%
 & &&$h$ &  & $0$ \\
 &$\ast$& && $\ddots$ & \\
& &&$0$&  & $h$ \\\hline \rule{0pt}{1.1\normalbaselineskip}
$H_1$& $\dots$ &
          $H_{d}$ &$0$& $\dots$ &
          $0$  %
\end{tabular}\right)$ }%
\]
\vskip0.2cm

\noindent
The row operations correspond to subtracting $\bK[x_1,\dots,x_n]$-linear 
combinations of $\tilde{f_1},\dots,\tilde{f}_{r}$ from $\hat{f}$. In this way, we get a polynomial 
$\tilde{f}$ as desired: The images of $\tilde{f}$ and $f$ in ${\mathcal O}_{W}(W\cap D(h))$ coincide,
and for $i=d+1,\dots, n$, the $\frac{\partial \tilde{f}}{\partial x_i}$ are
mapped to zero in ${\mathcal O}_{W}(W\cap D(h))/\langle f \rangle$.
In fact, we have
\begin{equation}
\label{ref:def-hs}
(\frac{\partial \tilde{f}}{\partial x_1},\ldots,\frac{\partial \tilde{f}}{\partial x_n})
=(H_1,\ldots,H_d,0,\ldots,0)
\end{equation}
as an equality over  ${\mathcal O}_W (W\cap D(h))/\langle f\rangle$.
\end{construction}

\begin{lem}
\label{chain rule} 
With notation as above, consider the extended ideal 
$$
\Delta(f) =\langle f, H_1,\dots, H_d \rangle {\mathcal O}_{W}(W\cap D(h)).
$$
Then 
$$\Delta(f)\;\! \widehat{\mathcal{O}_{W,p}}=\Delta_p(f) \;\text{ for each point }\; p\in W\cap D(h).$$
\end{lem}
\begin{proof}
Let a point $p=(a_1,\dots, a_n)\in W \cap D(h)$ be given.  Write
$\boldsymbol{x-a}=\{x_1-a_1, \dots, x_n-a_n\}$ and $\boldsymbol x=\{x_1, \dots, x_n\}$. Then
$$\widehat{{\mathcal O}_{W,p}} \cong
\bK[[\boldsymbol{x-a}]]/I_{W}\bK[[\boldsymbol{x-a}]],$$ 
and the natural map 
$$\Psi:  \bK[\boldsymbol x]\longrightarrow \bK[[\boldsymbol{x-a}]]\longrightarrow  
\bK[[\boldsymbol{x-a}]]/I_{W} \bK[[\boldsymbol{x-a}]]$$
factors through the inclusion 
$\bK[W] \rightarrow \widehat{{\mathcal O}_{W,p}}$. Moreover, by our assumptions in 
Notation \ref{notation:pos}, the isomorphism of the Cohen structure theorem reads
\begin{eqnarray*}
\bK[[y_1,\dots,y_d]] & \overset{\Phi}{\longrightarrow} & \bK[[\boldsymbol{x-a}]]/I_{W}\bK[[\boldsymbol{x-a}]],\\
 y_i         & \longmapsto & x_i-a_i. \;\; 
\end{eqnarray*}
The inverse isomorphism $\Phi^{-1}$ is of type
\begin{eqnarray*}
 y_i         & \longleftarrow\!\shortmid & x_i-a_i\;\; \text{ if } 1 \leq i \leq d,\\
 m_i(y_{1},\dots,y_{d}) & \longleftarrow\!\shortmid & x_i-a_i\;\;
                                          \text{ if } d+1 \leq i \leq n.
\end{eqnarray*}
Then ${\Phi^{-1}}\circ \quad \!\!\!\! \Psi$ is the map 
$$g\mapsto g(\boldsymbol {y+a}', m(\boldsymbol{y})+\boldsymbol{a}'')),$$
where $\boldsymbol{a}'=\{a_1,\dots, a_d\}$, $\boldsymbol{a}''
=\{a_{d+1},\dots, a_n\}$, and $\boldsymbol y=\{y_1, \dots, y_d\}$.
Hence, for each $g\in \bK[\boldsymbol x]$, the vector of partial derivatives
$$
\begin{pmatrix} \frac{\partial \Phi^{-1}(\Psi(g))}{\partial y_1}, & \dots, &
                   \frac{\partial \Phi^{-1}(\Psi(g))}{\partial y_d}
   \end{pmatrix}
$$
is obtained as the product
$$
\scalebox{0.9}{ $\begin{pmatrix} \frac{\partial g}{\partial x_1}
              (\boldsymbol {y+a}', m(\boldsymbol{y})+\boldsymbol{a}''), & \dots, &
                   \frac{\partial g}{\partial x_n}
               (\boldsymbol {y+a}', m(\boldsymbol{y})+\boldsymbol{a}'')
   \end{pmatrix} \cdot
   \begin{pmatrix} & E_d & \smallskip \\ \hline \rule{0pt}{1.1\normalbaselineskip}
                   \frac{\partial m_{d+1}}{\partial y_1} & \dots &
                   \frac{\partial m_{d+1}}{\partial y_d} \\ \vdots && \vdots  \\
                  \frac{\partial m_n}{\partial y_1} & \dots &
                   \frac{\partial m_n}{\partial y_d} 
  \end{pmatrix}$}
$$
(apply the chain rule). 
Taking $g=\tilde{f}$ with $\tilde{f}$ as in Construction \ref{con deriv}, 
we deduce from Equation \eqref{ref:def-hs}  that
$$
\frac{\partial \Phi^{-1}(\Psi(\tilde{f}))}{\partial y_j} = \Phi^{-1}(\Psi(H_j))
$$
as an equality over $\bK[[y_1,\ldots,y_d]]/\langle \Phi^{-1}(\Psi(f))\rangle$, for $j=1,\dots, d$. The
result follows by applying $\Phi$ since $\Psi(\tilde{f})=\Psi(f)$ by the very construction of $\tilde{f}$.
\end{proof}

\begin{notation}
In the situation of Lemma \ref{chain rule}, motivated by the lemma and its proof, we write
$$\frac{\partial f}{\partial X_{j}}:=H_j\in \bK[x_1,\dots,x_n], \;\text{ for }\; j=1,\dots, d.$$
\end{notation}

Summing up, we get:

\begin{lem}
\label{locus order 2}
Let $f_{r+1},\dots,f_s\in \bK[x_1,\dots, x_n]$ represent a set of generators for the vanishing ideal
$I_{X,W}$. Then $\Sing(I_{X,W},2)\cap D(h) $ is the locus
\[
V \left(I_X+\left\langle \frac{\partial f_{i}}{\partial X_{j}}\;\!  \middle| \;\! r+1 \leq i \leq s,
 1 \leq j \leq d \right\rangle \right) \cap D(h)
\]
which is computable by the recipe given in Construction \ref{con deriv}.
\end{lem}

If the intersection of $\Sing(I_{X,W},2)$ with one principal open set from a covering as in Lemma 
\ref{def deriv} is non-empty, then $X$ is singular by Lemma \ref{lem: criterion-singular-II}, and
our smoothness test terminates. If all these intersections are empty we
iterate our process:

\begin{lem}
[\cite{BEV}]\label{lem max contact} 
Let $f_{r+1},\dots,f_{s}\in{\mathbb{K}}[x_{1},\dots,x_{n}]$ represent a set of generators 
for the vanishing ideal $I_{X,W}$.  Retaining Notation \ref{notation:pos},
suppose that $\Sing(I_{X,W},2)\cap D(h)=\emptyset.$
Then there is a finite covering  of $X\cap D(h)$ by principal open subsets
of type $D(h\cdot g)\subset\mathbb{A}_{\bK}^{n}$ such that:
\begin{enumerate}
\item Each polynomial $g$ is a derivative $\frac{\partial f_{i}}{\partial X_{j}}$ of some $f_{i}$,
$r+1\leq i \leq n$.
\item If we set $W^{\prime}=V(f_{1},\dots,f_{r},f_{i})\subset \mathbb A^n_{\bK}$, then 
$W'\cap D(h\cdot g)$ is a smooth complete intersection of codimension $r+1$ in $D(h\cdot g)$.
\item We have $X\cap D(h\cdot g)\subset W^{\prime}\cap D(h\cdot g)$.
\end{enumerate}
\end{lem}

\begin{proof}
Let $p_0\in X \cap D(h)$. Then, since $\Sing(I_{X,W},2)\cap D(h)=\emptyset$ 
by assumption, the order of $f_{i}$ is $\ord_{p_0}(f_{i})=1$ for at least one $i\in\{r+1,\dots,s\}$.
Equivalently, one of the partial derivatives of $f_i$, say $\frac{\partial f_{i}}{\partial X_{j}}$,
does not vanish at $p_0$. Then, if we set $g=\frac{\partial f_{i}}{\partial X_{j}}$ and
$W^{\prime}=V(f_{1},\dots,f_{r},f_{i})$, properties  (1) and (3) of the lemma are clear by construction.
With regard to (2), again by construction, we have $\ord_{p}(f_{i})=1$ for each $p\in D(h\cdot g)$.
This implies for each $p\in D(h\cdot g)$:
\begin{itemize}
\item[(a)] We have $\nu^*(W',p) = (1,..,1,1)\in {\mathbb N}^{r+1}$.
\item[(b)] The image of $f_i$ in $\mathcal O_{W,p}$ is a non-zero non-unit.
\end{itemize}
Then, $W'\cap D(h\cdot g)$ is smooth by (a) and Hironaka's Criterion
\ref{crit-Hir}. Furthermore, each local ring $\mathcal O_{W,p}$ is regular and, thus, 
an integral domain. Hence, by (b) and Krull's principal ideal theorem, the ideal
generated by the image of $f_i$ in $\mathcal O_{W,p}$ has codimension 1.
We conclude that $W'\cap D(h\cdot g)$ is a complete intersection of codimension 
$r+1$ in $D(h\cdot g)$.

The result follows since the Zariski topology is quasi-compact.
\end{proof}

\begin{rem}
\label{rem-get-rad-ideal}
In the situation of the proof above, Hironaka's criterion actually
allows us to conclude that the affine scheme $$\Spec \left(\mathcal O_{\mathbb A^n_{\bK}}(D(h\cdot g))/
\langle f_1,\dots, f_{r},f_{i}\rangle O_{\mathbb A^n_{\bK}}(D(h\cdot g))\right)$$
is smooth. In particular, $\langle f_1,\dots, f_{r},f_{i}\rangle  O_{\mathbb A^n_{\bK}}(D(h\cdot g))$
is a radical ideal.
\end{rem}

\begin{rem}
At each iteration step of our process, we start from embeddings of type $X\cap D(q)\subset
W\cap D(q)\subset \mathbb A^n_{\bK}$ rather than from an embedding $X\subset 
W\subset \mathbb A^n_{\bK}$. This is not a problem: When we use
the trick of Rabinowitch to regard $X\cap D(q)\subset W\cap D(q)$ as affine varieties 
in $\mathbb A^{n+1}_{\bK}$, and apply Lemma~\ref{lem covering by minors} in 
$\mathbb A^{n+1}_{\bK}$, the extra variable in play will not 
appear in the local systems of parameters.
\end{rem}

\begin{rem}[The Role of the Ground Field]
\label{rem:k-versus-K}
Our algorithms essentially rely on Gr\"obner basis techniques (and not, for example, on polynomial factorization). 
While the geometric interpretation of what we do is concerned with an algebraically closed
field $\bK$, the algorithms will be applied  to ideals which are defined over a 
subfield $\bk\subset \bK$ whose arithmetic can be handled by a computer. This makes
sense since any  Gr\"obner basis of an ideal $J\subset \bk[x_1,\dots, x_n]$ is also a Gr\"obner 
basis of the extended ideal $J^e=J\bK[x_1,\dots, x_n]$. Indeed, if $J$ is given by generators with coefficients 
in $\bk$, all computations in Buchberger's Gr\"obner basis algorithm are carried through over $\bk$. 
In particular, if a property of ideals can be checked using Gr\"obner bases, then $J$ has this 
property iff $J^e$ has this property. For example, if $J$ is equidimensional, then $J^e$ is 
equidimensional as well. Or, if the condition asked by the Jacobian criterion is fulfilled for $J$,
then it is also fulfilled for $J^e$.

The standard reference for theoretical results on extending the ground field is \cite[VII, \S 11]{ZSCA}.
To give another example, if $\bk$ is perfect, and $J$ is a radical ideal, then $J^e$ is a radical ideal, too. 

\end{rem}

\begin{notation}
In what follows, $\bk\subset \bK$ always denotes a field extension with $\bk$ perfect
and $\bK$ algebraically closed. If $I\subset \bk[\boldsymbol x] =\bk[x_1,\dots, x_n]$ is an ideal, then $V(I)$
stands for the vanishing locus of $I$ in $\mathbb A^n(\bK)$. Similarly, if
$q\in  \bk[\boldsymbol x]$, then $D(q)$ stands for the principal open set  defined by $q$
in $\mathbb A^n(\bK)$.
\end{notation}

We are now ready to specify the smoothness test. We start from ideals 
$$I_W =\langle f_1,\dots, f_r\rangle\subset I_X =\langle f_1,\dots, f_s\rangle\subset \bk[\boldsymbol x]$$
and a polynomial $q \in \bk[\boldsymbol x]$ such that 
\begin{itemize}
\item[\phantom{$(\diamondsuit)$}]\quad $I_X$ is equidimensional and radical,
\item[$(\diamondsuit)$]\quad $I_W\;\!\mathcal O_{\mathbb A^n_{\bK}}(D(q))$ is a radical ideal of codimension $r$,
\item[\phantom{$(\diamondsuit)$}]\quad $V(I_W)\cap D(q)$ is smooth.
\end{itemize}
Note that under these conditions, $V(I_W)\cap D(q)\subset D(q)$ is a complete intersection.
The overall structure of our algorithm then consists of the following main components:

\begin{enumerate}[leftmargin=8mm]
\item \emph{Covering by principal open sets}. Find a set $L$ of $r\times r$ submatrices $M$ of the Jacobian matrix 
of $f_1,\dots,f_r$ such that all minors $\det(M)$ are non-zero, and such that%
\[
q \in \sqrt{ \left\langle f_{1},\ldots,f_{r}\right\rangle+ \left\langle\det(M)\mid M\in L\right\rangle}.%
\]

\item \emph{Local system of parameters}. By Lemma \ref{lem covering by minors}, for each $M\in L$, there 
is a consistent choice of a regular system of parameters associated to $M$. Without loss of 
generality, we can assume that $M$ involves the variables $x_{d+1},\ldots x_n$. 
Then we may choose the regular system of parameters to be induced by $x_{1},\ldots,x_d$.

\item \emph{Derivatives relative to the local system of parameters.} Find the matrix of cofactors $A$ of $M$ with $A\cdot M=\det(M)\cdot E_{r}$ and let%
\[
\widehat{F}:=\left(
\begin{array}
[c]{c}%
\tilde{f}_{1}\\
\vdots\\
\tilde{f}_{s}\\
\hat{f}_{s+1}\\
\vdots\\
\hat{f}_{r}%
\end{array}
\right)  =%
\begin{pmatrix}
A & 0\\
0 & \det(M)\cdot E_{s-r}%
\end{pmatrix}
\cdot\left(
\begin{array}
[c]{c}%
f_{1}\\
\vdots\\
f_{s}%
\end{array}
\right) \ .
\]
By Lemma \ref{locus order 2}, the locus of order $\geq2$ is empty if and only~if%
\[
q\in\sqrt{\left\langle f_{1},\ldots,f_{s},\partial f_{i}/\partial X_{j}\mid
i,j>r\right\rangle },%
\]
where, by Construction \ref{con deriv}, the ideal of the derivatives $\partial f_{i}/\partial X_{j}$ is generated by the entries of the left
lower block of the Jacobian matrix of $\widehat{F}$ after the row reduction\\%

\[\scalebox{0.85}{ 
$\mathcal{J}(\widehat{F})=\left(
\begin{tabular}
[c]{c|ccc}
& $\det(M)$ &  & $0$\\
$\ast$ &  & $\ddots$ & \\
& $0$ &  & $\det(M)$\\\hline
$\ast$ &  & $\ast$ &
\end{tabular}
\right)  \mapsto\left(
\begin{tabular}
[c]{c|ccc}
& $\det(M)$ &  & $0$\\
$\ast$ &  & $\ddots$ & \\
& $0$ &  & $\det(M)$\\\hline
$\ast$ & $0$ & $\cdots$ & $0$%
\end{tabular}
\right)$
}\]
\\

\item \emph{Descent in codimension of $X$.} Consider a representation%
\[
q^{m}=\sum\alpha_{i,j}\cdot\partial f_{i}/\partial X_{j}\operatorname{mod}%
\left\langle f_{1},\ldots,f_{s}\right\rangle. %
\]
Suppose $\alpha_{i,j}\neq0$ and $\partial f_{i}/\partial X_{j}\neq 0$. Then by Lemma \ref{lem max contact}, we can pass to a new variety $W' = W\cap V(f_i)\supset X$
and a new principal open set $D(q')$ with $$q' = q\cdot h\cdot\partial f_{i}/\partial X_{j},$$ such that $W'\cap D(q')\subset D(q')$ is a smooth complete intersection of codimension $r+1$.
In this way, we obtain a covering of $X\cap D(q)$ by principal open sets $D(q')$ with $X\cap D(q')\subset W'\cap D(q')$ and iterate.\medskip
\end{enumerate}

Algorithm \ref{alg smoothnesstest} \texttt{HybridSmoothnessTest} collects
the main steps of the smoothness test. 
It calls Algorithm \ref{alg deltacheck} \texttt{DeltaCheck} 
to check whether $\Sing(I_{X,W},2) \cap D(q) \not= \emptyset$. In this case, it 
returns \texttt{false} and terminates. Otherwise, it  calls 
Algorithm \ref{alg descendembsmooth} \texttt{DescentEmbeddingSmooth} which implements
Lemma \ref{lem max contact}. The next step is to recursively apply the \texttt{HybridSmoothnessTest} 
in the resulting embedded situations.
If the codimension reaches a specified value, the algorithm invokes a relative version of the Jacobian 
criterion by calling Algorithm \ref{alg embJac} \texttt{EmbeddedJacobian}. 

\begin{algorithm}[h]
\caption{\texttt{HybridSmoothnessTest}}
\label{alg smoothnesstest}
\begin{algorithmic}[1]
\REQUIRE   Ideals $I_W =\langle f_1,\dots, f_r\rangle\subset I_X =\langle f_1,\dots, f_s\rangle\subset \bk[\boldsymbol x]$
and a polynomial $q \in \bk[\boldsymbol x]$ such that $(\diamondsuit)$ holds; a non-negative integer $c$.
\ENSURE  {\texttt true} if $V(I_X \cap D(q))$ is smooth, {\texttt false} otherwise.\\
\IF {$\operatorname{dim}(I_W)-\operatorname{dim}(I_X)$ = 0}
\RETURN {\texttt true}
\ENDIF
\IF {$\operatorname{dim}(I_W)-\operatorname{dim}(I_X) \leq c$}
\RETURN {\texttt{EmbeddedJacobian}($I_W$,$I_X$)}
\ENDIF
\IF {{\bf not} $\operatorname{\texttt{DeltaCheck}}(I_W,I_X,q)$}
\RETURN {\texttt false}
\ENDIF\label{Jac1}
\STATE $L=\operatorname{\texttt{DescentEmbeddingSmooth}}(I_W,I_X,q)$\label{Jac2}
\FORALL {($I_{W'},I_{X},q')\in L$}
\IF {{\bf not} $\texttt{HybridSmoothnessTest}(I_{W'},I_{X},q',c)$ }
\RETURN {\texttt false}
\ENDIF
\ENDFOR
\RETURN {\texttt true}
\end{algorithmic}
\end{algorithm}

\begin{algorithm}[h]
\caption{\texttt{DeltaCheck} for an affine chart}
\label{alg deltacheck}
\begin{algorithmic}[1]
\REQUIRE   Ideals $I_W =\langle f_1,\dots, f_r\rangle\subset I_X =\langle f_1,\dots, f_s\rangle\subset \bk[\boldsymbol x]$ 
and a polynomial $q \in \bk[\boldsymbol x]$ such that $(\diamondsuit)$ holds.

\ENSURE  {\texttt true} if $\Sing(I_{X,W},2) \cap D(q) = \emptyset$, {\texttt false} otherwise.\smallskip \\

\hspace{-0.5cm}\hbox{\{ \emph{First handle the case $I_W=\langle 0\rangle$, $q=1$; then
$x_1,\dots,x_n$  induce \}}} \\ \hspace{-0.5cm}\hbox{\{ \emph{a local system of parameters at every point of $W$ \}}}
\IF {$I_W=\langle 0\rangle$ {\bf and }$q=1$}
\IF {$1 \in \langle f_1,\dots,f_s, \frac{\partial f_1}{\partial x_1},\dots, \frac{\partial f_s}{\partial x_n}\rangle$}
\RETURN {\texttt true}
\ELSE
\RETURN {\texttt false}
\ENDIF
\ENDIF

\hspace{-0.5cm}\hbox{\{\emph{ Initialization \}}}
\STATE $Q=\langle 0\rangle$
\STATE $L_1  = \{ r\times r-{\rm submatrices\; } M \;{\rm of }\;\operatorname{Jac}(I_W)\; \mid \operatorname{det}(M)\neq0 \operatorname{mod} I_X\}$

\hspace{-0.5cm}\hbox{\{\emph{ Main Loop: Cover by complements of minors \}}}
\WHILE {$L_1\neq \emptyset$ {\bf and} $q \not\in Q$}\label{line while delta}
\STATE choose $M \in L_1$
\STATE $L_1= L_1 \setminus \{M\}$
\STATE $q_{new}=\operatorname{det}(M)$
\STATE $Q=Q+\langle q_{new}\rangle$
\STATE compute the $r\times r$ cofactor matrix $A$ with
$$A \cdot M = q_{new} \cdot \operatorname{Id}_{r}$$

\hspace{-1.cm}\hbox{\{\emph{ Test $\Sing(I_{X,W},2) \subset V(q_{new}) \cup V(q)$ \}}}
\STATE $C_M= I_X +  J_X$, where\\
$J_X  =  \left\langle q_{new} \cdot \frac{\partial f_i}{\partial x_j}
-\sum\limits_{k \text{ column of }M \atop
l\text{ row of } M}
\frac{\partial f_l}{\partial x_j} A_{lk}
\frac{\partial f_i}{\partial x_k} \hspace{2mm}\middle|\hspace{2mm}
 \substack{ r+1 \leq i \leq s\smallskip \\j \text{ not a column of } M} \right\rangle$
\label{line derivatives}
\IF {$q_{new}\cdot q \not\in \sqrt{C_M}$}
\RETURN \texttt false
\ENDIF
\ENDWHILE
\RETURN {\texttt true}
\end{algorithmic}
\end{algorithm}

\begin{rem}
Modifying the approach discussed in the components (3) and (4) above, Algorithm \ref{alg descendembsmooth} computes 
the products $h \cdot \frac{\partial f_i}{\partial X_j}$ 
directly as appropriate $(r+1) \times (r+1)$ minors of the Jacobian matrix in step 5 of Algorithm \ref{alg descendembsmooth}. 
This exploits the well-known fact that 
subtracting multiples of one row from another one as in Gaussian elimination does not change the determinant of a square matrix 
-- or in our case the maximal minors 
of the $(r+1) \times n$ matrix.

In practical applications, it can be useful to first check whether there is an $ r \times r$ minor, say $N$, of the Jacobian matrix 
of $I_W$ which divides $q$. If this is the case, we can restrict in step 6 of Algorithm \ref{alg descendembsmooth} to those minors in $I_{min,i}$ which involve $N$, since these minors already form a generating system of~$I_{min,i}$.

Combining these two 
modifications provides an important enhancement with regard to efficiency of the hybrid smoothness test as presented in \cite{smoothtst}.
\end{rem}

\begin{algorithm}[h]
\caption{\texttt{DescentEmbeddingSmooth}}
\label{alg descendembsmooth}
\begin{algorithmic}[1]
\REQUIRE   Ideals $I_W =\langle f_1,\dots, f_r\rangle\subset I_X =\langle f_1,\dots, f_s\rangle\subset \bk[\boldsymbol x]$ 
and a polynomial $q \in \bk[\boldsymbol x]$  such that $(\diamondsuit)$ and $D(q)\cap \operatorname{Sing}(I_X,2)=\emptyset$ hold.

\ENSURE Triples $(I_{W_i},I_{X}, q_i)$ 
such that $I_{W_i}\subset I_{X}$ together with $q_i$ satisfy $(\diamondsuit)$, and such that
              $V(I_X) \cap D(q) \subset \bigcup_i (V(I_{W_i}) \cap D(q_i))$.  

\smallskip
\hspace{-0.5cm}\hbox{\{ \emph{Direct descent: no need to find an open covering of $V(I_X)\cap D(q)$} \}}
\IF { $\Sing(I_{V(f_i),W},2) \cap  D(q) = \emptyset$ and $q \notin \sqrt{\langle f_1,\ldots,f_r,f_i \rangle} $ for some $i \in \{r+1,\ldots,n\}$}

        \STATE $I_{W_1}=\langle f_1,\dots,f_r,f_i\rangle$
        \RETURN \{($I_{W_1}$,$I_{X}$,$q$)\}
     \ENDIF

\hspace{-0.5cm}\hbox{\{ \emph{Descent by constructing an open covering of $V(I_X)\cap D(q)$} \}}
     \FOR {$i \in \{r+1,\ldots,s\}$}
         \STATE $I_{min,i}=$ $\langle (r+1)\times (r+1)$ minors of the Jacobian matrix of
                $f_1,\ldots,f_r,f_i\rangle$\\
      \ENDFOR \smallskip
      \STATE \label{line cover dec} from among the generators of the ideals
             $I_{min,1},\ldots,I_{min,s}$, find minors $h_1,\ldots,h_t \neq 0 \operatorname{mod} I_X$ such that
             $q \in \sqrt{I_X+\langle h_1,\ldots,h_t \rangle}$   \\
      \STATE fix $i_1,\ldots,i_t \in \{r+1,\ldots s\}$
             with $h_j \in I_{min,i_j}$\\
      \FOR {$j=1,\ldots,t$}
          \STATE $I_{W_j}=\langle f_1,\ldots,f_r,f_{i_j} \rangle$
      \ENDFOR
\RETURN $\{(I_{W_1},I_{X},q\cdot h_1),\ldots,(I_{W_t},I_{X},q\cdot h_t)\}$

\end{algorithmic}
\end{algorithm}

\begin{algorithm}[h]
\caption{\texttt{EmbeddedJacobian}}
\label{alg embJac}
\begin{algorithmic}[1]
\REQUIRE   Ideals $I_W =\langle f_1,\dots, f_r\rangle\subset I_X =\langle f_1,\dots, f_s\rangle\subset \bk[\boldsymbol x]$ 
and a polynomial $q \in \bk[\boldsymbol x]$ such that $(\diamondsuit)$ holds.

\ENSURE  {\texttt true} if $V(I_X)\cap D(q)$ is smooth, {\texttt false} otherwise.\\
\STATE $Q=\langle 0\rangle$
\STATE \label{line Jac M} $L = \{ r\times r-{\rm submatrices\; } M \;{\rm of }\;\operatorname{Jac}(I_W)\; \mid \operatorname{det}(M)\neq0\operatorname{mod} I_X\}$

\hspace{-0.5cm}\hbox{\{\emph{ Read off regular system of parameters for non-trivial $h$ \}}}
\IF {$\operatorname{det}(M)\! \mid \! q$ for some $M \in L $}
\STATE delete all other elements from $L$
\ENDIF

\hspace{-0.5cm}\hbox{\{\emph{ Covering by complements of the minors \}}}
\WHILE {$L\neq \emptyset$ {\bf and} $q \not\in Q$}\label{line while jac}
\STATE choose $M \in L$
\STATE $L= L \setminus \{M\}$
\STATE \label{line det jac}$q_{new}=\operatorname{det}(M)$
\STATE $Q=Q+\langle q_{new}\rangle$
\STATE \label{line adj jac}compute the $r\times r$ cofactor matrix $A$ with
$$A \cdot M = q_{new} \cdot \operatorname E_{r}$$

\hspace{-1.0cm}\hbox{\{\emph{ Jacobian matrix of $I_X$ w.r.t. local system of parameters for $I_W$ \}}}
\STATE $Jac =
\left(q_{new} \cdot \frac{\partial f_i}{\partial x_j}
-\sum\limits_{k \text{ a column of }M \atop
l\text{ a row of } M}
\frac{\partial g_l}{\partial x_j} A_{l,k}
\frac{\partial f_i}{\partial x_k} \right)\in \bk[\boldsymbol x]^{(s-r)\times (n-r)}$\\
where $r+1\leq i\leq s$ and $j$ is not a column of $M$ \medskip
\STATE $c=\codim\left(V(I_X)\cap D(q)\subset V(I_W)\cap D(q)\right)$
\STATE $J = I_X + I_{m}$, where
       $I_{m}=\langle
       c\times c-$minors of $Jac$ $\rangle$
\IF {$q_{new} \cdot q \not\in \sqrt{J}$}\label{line false jac}
\RETURN  \texttt false
\ENDIF
\ENDWHILE
\RETURN {\texttt true}
\end{algorithmic}
\end{algorithm}

\begin{rem} 
In explicit experiments, we typically arrive at an ideal $I_X$ by using a specialized construction method
which is based on geometric considerations. From these considerations, some properties of $I_X$ 
might be already known. For example, it might be clear that $V(I_X)$ is irreducible. Then there
is no need to check the equidimensionality of $I_X$. 
If we apply the algorithm without testing
whether $I_X$ is radical, and the algorithm returns {\tt true}, then $I_X$ must be radical.
This is clear from the Jacobian criterion and the fact that Hironaka's criterion checks smoothness 
in the scheme theoretical sense (see Remark \ref{rem-get-rad-ideal}).
\end{rem}

\begin{rem} 
If we do not have some specific pair $(I_W,q)$ in mind, we can always start Algorithm \ref{alg smoothnesstest} with $(I_W,q)=(\langle 0 \rangle, 1)$. In this case the algorithm determines smoothness of whole affine variety $V(I_X)\subset \mathbb A^n(\bK)$.
\end{rem}

\section{Petri Nets and the GPI-Space environment}\label{sec GPI}

\subsection{GPI-Space and Task-Based Parallelization} 
GPI-Space \cite{GPI} is a task-based workflow management system for high performance environments. It is based on David Gelernter's 
approach of separating coordination and computation \cite{gelernter} which leads to the explicit visibility of dependencies, 
and is beneficial in many aspects. We illustrate this by discussing some of the concepts realized in GPI-Space:

\vskip0.1cm
\begin{compactitem}[leftmargin=4mm]
\item The coordination layer of GPI-Space uses a separate, specialized language, namely Petri nets \cite{petri},
which leaves optimization and rewriting of coordination activities to experts for data 
management rather than bothering experts for computations in a particular domain of 
application (such as algebraic geometry) with these things. 
\item Complex environments remain hidden from the domain experts and are managed automatically. This includes automatic parallelization, 
automatic cost optimized data transfers and latency hiding, automatic adaptation to dynamic changes in the environment, and resilience.
\item Domain experts can use and mix arbitrary implementations of their algorithmic solutions. For the experiments done for this paper, 
GPI-Space manages several (many) instances of {\sc Singular} and, at the same time, other code written in \verb!C++!.
\item The use of virtual memory allows one not only to scale applications beyond the limitations imposed by a single machine, but also to
couple legacy applications that normally can only work together by writing and reading files. Also, the switch between low latency, low capacity memory (like DRAM) and high latency, high capacity memory (like a parallel file system) can be done without changing the application.
\item Optimization goals like ``minimal time to solution'', ``maximum throughput'', or ``minimal energy consumption'' are achieved independently 
from the domain experts' implementation of their core algorithmic solutions.
\item Patterns that occur in the management of several applications are explicitly available and can be reused. Vice versa, 
computational core routines can be reused in different management schemes. Optimization on either side is beneficial 
for all applications that use the respective building blocks. 
\end{compactitem}

\vskip0.1cm
\noindent
GPI-Space consists of three main components:
\begin{compactitem}[leftmargin=4mm]
\item A distributed, resilient, and scalable runtime system for huge dynamic environments that is responsible for managing the available resources, 
specifically the memory resources and the computational resources. 
The scheduler of the runtime system assigns activities to resources with respect to both the needs 
of the current computations and the overall optimization goals.
\item A Petri net based workflow engine that manages the full application state and is responsible for automatic parallelization and dependency tracking.
\item A virtual memory manager that allows different activities and/or external programs to communicate and share partial results. 
The asynchronous data transfers are managed by the runtime system rather than the application itself, and synchronization is done in a
way that aims at hiding latency.
\end{compactitem}

\vskip0.1cm
Of course, the above ideas are not exclusive to GPI-Space -- many other systems exist that follow similar strategies. In the last few years, 
task-based programming models are getting much attention in the field of high performance computing. They are realized in systems such as
OmpSs \cite{ompss}, StarPU \cite{starpu}, and PaRSEC \cite{parsec}. All these systems have in common that they do explicit data 
management and optimization in favor of their client applications. The differences are in their choice of the coordination language, in 
their choice of the user interface, and in their choice of how general or how specific they are.
It is widely believed that task-based systems are a promising approach to program the current and upcoming very large and very complex 
super computers in order to enable domain experts to get a significant
fraction of the theoretical peak performance \cite{dongarra, sppexa}.

As far as we know, there have not yet been any attempts to use systems originating from high performance numerical simulation
in the context of computational algebraic geometry, where the main workhorse is Buchberger's algorithm for computing Gr\"obner bases. 
Although this algorithm performs well in many practical examples of interest, its worst case complexity is doubly exponential 
in the number of variables \cite{MM}. This seems to suggest that  algorithms in computational algebraic 
geometry are too unpredictable in their time and memory consumption for the successful integration 
into task-based systems. However, numerical simulation also encounters problems of unpredictability, and there is
already plenty of knowledge on how to manage imbalances imposed by machine jitter or different 
sizes of work packages. For example, numerical state of the art code to compute flows makes use of mesh adaptation.
This creates great and unpredictable imbalances in computational effort which are addressed on the fly
by the respective simulation framework. 

The high performance computing community aims for energy efficient computing, just because the machines they are using are so 
big that it would be too expensive to not make use of acquired resources. One key factor to achieve good efficiency is perfect load 
balancing. Another important topic in high performance computing is the non-intrusive usage of legacy code. GPI-Space is not only  
able to automatically balance, to automatically scale up to huge machines, or to tolerate machine failures, it can also use existing 
legacy applications and integrate them, without requiring any change to them. This turned out to be the great door opener for 
integrating {\sc Singular}  into GPI-Space. In fact, in our applications, GPI-Space manages several (many) instances of {\sc Singular} in its existing 
binary form (without any need for changes). 

Our first experience indicates that the 
tools used in high performance computing are mature, both, in terms of operations and in 
terms of capabilities to manage complex applications from symbolic computation.
We therefore believe that it is the right time to apply these tools to domains such as computational algebraic geometry.

In GPI-Space, the coordination language is based on Petri nets, which are known to be a good choice because 
of their graphical nature, their locality (no global state), their concurrency (no events, just dependencies), and 
their reversibility (recomputation in case of failure is possible) \cite{aalst}. Incidentally, these are all properties 
that Petri intentionally borrowed from physics for the use in computer science \cite{brauer}. Moreover, Petri nets 
share many properties with functional languages, especially their well-known advantages of modularity and 
direct correspondence to algebraic structures, which qualify them as both powerful and user friendly 
\cite{backus, hughes}.

The following section describes in more detail what Petri nets are and why they are a good choice to describe dependencies.

\subsection{Petri Nets}\label{sec petrinet}

In 1962, Carl Adam Petri proposed  a formalism to describe concurrent asynchronous systems \cite{petri}. 
His goal was to describe systems that allow for adding resources to running computations 
without requiring a global synchronization, and he discovered an elegant solution that connects resources with other resources
only locally. Petri nets are particularly interesting since they have the  following properties:
\begin{compactitem}
\item They are graphical (hence intuitive) and hierarchical (so that applications can be decomposed into building blocks that are Petri nets themselves).
\item They are well-suited for concurrent environments since there are no events that require a (total) ordering. Instead,
Petri nets are state-based and describe at any point in time the complete state of the application. That locality (of dependencies)
also allows one to apply techniques from term-rewriting to improve (parts of) Petri nets in their non-functional properties, 
for example to add parallelism 
or checkpointing. 
\item They are reversible and enable backward computation: If a failure causes the loss of a partial 
result, it is possible to determine a minimal set of computations whose repetition will recover the 
lost partial result.
\end{compactitem}
\vskip0.1cm

The advantages of Petri nets as a mathematical modeling language have been summarized very nicely by van der Aalst \cite{aalst}:
They have precise execution semantics that assign specific meanings to the net, serve as the basis for
an agreement, are independent of the tools used, and enable process analysis and solutions.
Furthermore, because Petri nets are not based on events but rather on state transitions, it is possible to
differentiate between activation and execution of an elementary functional unit. In particular, interruption
and restart of the applications are easy. This is a fundamental condition for fault tolerance to hardware
failure. Lastly, van der Aalst notes the availability of mature analysis techniques that besides proving the
correctness, also allow performance predictions.

\subsubsection{Formal Definitions and Graphical Representation}
\label{subsubsect:def-PN}

Petri nets generalize finite automata by complementing them with 
distributed states and explicit synchronization.  

\begin{df}
A \define{Petri net} is a triple $\K{P, T, F}$, where  $P$ and $T$ are disjoint finite sets, 
the sets of \define{places} respectively \define{transitions}, and where $F$ is a subset  
$F\subset\K{P\times T}\cup\K{T\times P}$,  the \define{flow relation} of the net. 
\end{df}

This definition addresses the static parts of a Petri net. In addition, there are dynamic aspects
which describe the execution of the net.

\begin{df} A \define{marking} of a Petri net $\K{P, T, F}$ is a function $M:P\to \NN$. 
If $M(p)=k$, we say that \define{$p$ holds $k$ tokens under $M$}. 
\end{df}

To describe a marking, we also write $M = \{(p,M(p)) \mid p\in P, M(p) \neq 0\}$.

\begin{rem}
For our purposes here, given a Petri net $\K{P, T, F}$ together with a marking $M$, we think of the transitions
as algorithms, while the tokens held by the places  represent the data (see Section \ref{subsubsect:ext} below for more on this).
Accordingly, given a place $p$ and a transition $t$, we say that $p$ is an \define{input} (respectively 
\define{output}) place of $t$ if $(p,t)\in F$ (respectively $(t,p)\in F$).
\end{rem}

A marking $M$ defines the \define{state} of a Petri net. We say that $M$ \define{enables} a transition $t$
and write 
$M\overset{t}\longrightarrow$, 
if all input places of $t$ hold tokens, that is,  $\K{p,t}\in F$ implies $M(p)> 0$. 
A Petri net equipped with a marking $M$  is executed by \define{firing} a single transition 
$t$ enabled by $M$. This means to consume a token from each input place of $t$, and to add
a token to each output place of $t$. In other words, the firing of $t$ leads to a new marking $M'$, with 
$M'(p)=M(p)-\card{\set{\K{p,t}}\cap F}+\card{\set{\K{t,p}}\cap F}$  for all $p\in P$.
Accordingly, we write 
$M\overset{t}\longrightarrow M'$, and say that $M'$ is \define{directly reachable} from $M$ (by firing $t$). 
Direct reachability defines the (weighted) \define{firing relation
  $R\subseteq \mathcal{M}\times T\times \mathcal{M}$}
over all markings $\mathcal{M}$ by $\K{M,t,M'}\in R\iff M\overset{t}\longrightarrow
M'$.
More generally, we say that a marking $M'$ is \define{reachable by $\hat{t}$} from a marking $M$ if there is a
\define{firing sequence $\hat{t}=t_0\cdots{}t_{n-1}$}  such that $M=M_0\overset{t_0}\longrightarrow M_1\overset{t_1}\longrightarrow\dots\overset{t_{n-1}}\longrightarrow M_{n-1}=M'$.
The corresponding graph is called the \define{state graph}.
Fundamental problems concerning state graphs
such as \define{reachability} or \define{coverability} have been subject to many studies, and effective methods have been
developed to deal with these problems \cite{brauer,mayr_reach,lambert_reach,priese_wimmel,karp_miller,kosaraju}.

The static parts of a Petri net are graphically represented by a bipartite directed graph as indicated in the two examples 
below. 
In such a graph, a marking is visualized
by showing its tokens as dots in the circles representing the places. See Section \ref{subsubsect:ext} 
for examples.

\begin{Example}[Data Parallelism in a Petri Net]
\label{ex:1}
The Petri net $\Phi=\K{P,T,F}$ with $P=\set{i,o}$, $T=\set{t}$ and $F=\set{\K{i,t},\K{t,o}}$ is depicted by the graph
\begin{petrinet}
  \node[place]      (0)              {$i$};
  \node[transition] (1) [right of=0] {$t$};
  \node[place]      (2) [right of=1] {$o$};
  \path[->]
        (0) edge (1)
        (1) edge (2) 
  ;
\end{petrinet}
Suppose we are given the marking $M=M_0=\set{\K{i,n}}$ for some $n>0$. Then $t$ is enabled by $M_0$, and firing $t$ 
means to move one token from $i$ to $o$. This leads to the new marking $M_1=\set{\K{i,n-1},\K{o,1}}$. 
Now, if $n>1$, the marking $M_1$ enables $t$ again, and $\Phi$ can fire until the marking 
$M'=M_n=\set{\K{o,n}}$ is reached. We refer to this by writing 
$M\overset{t^n}\longrightarrow M'$. 
Note that with this generalized firing relation, the $n$ incarnations of $t$ have no relation to each other -- conceptually, they fire 
all at the same time, that is, in parallel. This is exploited in GPI-Space, and makes much sense if we take into consideration that in 
the real world, the transition $t$ would need some time to finish, rather than fire immediately (see Section \ref{subsubsect:ext}
below for how to model time in Petri nets). Data parallelism is nothing else than splitting data into parts and applying the same given function 
to each part. This is exactly what happens here: Just imagine that each token in place $i$ represents some part of the data. 
\end{Example}
\begin{Example}[Task Parallelism in a Petri Net]
\label{ex:2}
Let $\Psi$ be the Petri net depicted by the graph
\begin{petrinet}
  \node[place]      (0)              {$i$};
  \node[transition] (1) [right of=0] {$s$};
  \node[invisible]  (i) [right of=1] {};
  \node[place]      (2) [above of=i,node distance=0.5\nd] {};
  \node[place]      (3) [below of=i,node distance=0.5\nd] {};
  \node[transition] (f) [right of=2] {$f$};
  \node[transition] (g) [right of=3] {$g$};
  \node[place]      (4) [right of=f] {$l$};
  \node[place]      (5) [right of=g] {$r$};
  \node[transition] (j) [right of=i,node distance=3\nd] {$j$};
  \node[place]      (6) [right of=j] {};
  \path[->]
        (0) edge (1)
        (1) edge (2)
        (1) edge (3)
        (2) edge (f)
        (f) edge (4)
        (3) edge (g)
        (g) edge (5)
        (4) edge (j)
        (5) edge (j)
        (j) edge (6)  
  ;
\end{petrinet} 
and consider the marking $M=\set{\K{i,1}}$. Then $\Psi$ can fire $s$ and thereby enable 
$f$ \emph{and} $g$. So this corresponds to the situation where different independent algorithms ($f$ and $g$) are 
applied to parts (or incarnations) of data. Note that $f$ and $g$ can run in parallel. Just like for the net $\Phi$
from Example \ref{ex:1},  multiple tokens in place $i$ allow for parallelism of $s$ and thereby of $f$, $g$, and $j$ 
as well. With enough such tokens, we can easily find ourselves in a situation where $s$, $f$, $g$, and $j$ 
are all enabled at the same time (see again Section \ref{subsubsect:ext}
below for the concept of time in Petri nets).
\end{Example}

To sum this up: Petri nets have the great feature to automatically know about \emph{all} activities that can be executed at any given time. 
Hence \emph{all} available parallelism can be exploited.

\subsubsection{Extensions of Petri Nets in GPI-Space}
\label{subsubsect:ext}

To model real world applications, the classical Petri net described above needs to be enhanced,
for example to allow for the modeling of time and data. This leads to extensions such as \define{timed} 
and \define{coloured Petri nets}. Describing these and their properties in detail goes beyond the scope of this article. 
We briefly indicate, however, what is realized in GPI-Space.

{\bf Time.} In the real world, transitions need time to fire (there is no concept of time in the classical Petri net).
In systems modeling, \define{timed Petri nets} are used to predict best or worst case running times. In  \cite{heiner_popova},
for example, the basic idea behind including time is to split the firing process into 3 phases:
\begin{enumerate}
\item The tokens are removed from the input places when a transition fires,
\item the transition holds the tokens while working, and
\item the tokens are put into the output places when a transition finishes working.
\end{enumerate}
This implies that a marking as above alone is not enough to describe
the full state of a timed Petri net. In addition, assuming that phases 1 and 3 
do not need any time, the description of such a state includes the knowledge of all active transitions 
in phase 2 and all tokens still in use. Passing from a standard to a timed Petri net, the behavior of the net is unaffected in the 
sense that any state reached by the timed net is also reachable with
the standard net (see again \cite{heiner_popova}).

{\bf Types and Type Safety.} 
As already pointed out, in practical applications, tokens are used to represent data.  In the classical Petri net, however, tokens 
carry no information, except that they are present or absent. It is therefore necessary to extend the classical concept
by allowing tokens with attached data values, called the \define{token colours} (see \cite{CPN}). Formally, in addition 
to the static parts of the classical Petri net, a \define{coloured Petri net} comes equipped
with a finite set $\Sigma$ of \define{colour sets}, also called \define{types},
together with a \define{colour function} $C: P\rightarrow \Sigma$ (``all tokens in a given place $p\in P$ represent
data of the same type'').
Now, a \define{marking} is not just a mapping  $P \rightarrow \NN$ (``the count of the tokens''), 
but a mapping  $\Delta \rightarrow\NN$, where $\Delta =\{(p,c) \mid p\in P, c\in C(p)\}$.
Imagine, for example, that the type $C(p)$ represents certain blocks of data. 
Then, in order to properly process the data stored in $c\in C(p)$, we typically need to know \emph{which block} 
out of \emph{how many blocks} $c$ is. That is, implementing the respective type means to equip each block
with two integer numbers. We will see below how to realize this in GPI-Space. 

Type safety is enforced in GPI-Space by rejecting Petri nets whose flow relation does not respect the imposed types.
More precisely, transitions are enriched by the concept of a \define{port}, which is a typed place holder for
incoming or outgoing connections.
Type safety is in general checked statically; for transitions relying on legacy code,
it is also checked dynamically during execution  (``GPI-Space does not trust legacy code'').

{\bf Expression Language.} 
GPI-Space includes an embedded programming language which serves a twofold purpose. On the one hand, it
allows for the introduction and handling of user-defined types. The type for blocks of data as discussed above,  for example,  may be described 
by the snippet
\begin{verbatim}
 <struct name="block">
   <field name="num" type="uint"/>
   <field name="max" type="uint"/>
 </struct>
\end{verbatim}
Types can  be defined recursively. Moreover, GPI-Space offers a special kind 
of transition which makes it possible to manipulate the colour of a token. In the above situation, for
instance, the ``next block`'' is specified by entering
\begin{verbatim}
 ${block.num} := ${block.num} + 1
\end{verbatim}
Again, all such expressions are type checked.

The second use made of the embedded language is the convenient handling of ``tiny
computations''. Such computations can be executed directly within the workflow engine rather than
handing them over to the runtime system for scheduling and execution, and 
returning the results to the workflow engine. 

{\bf Conditions.} 
In GPI-Space, the firing condition of a transition can be subject to a logical expression 
depending on properties of the input tokens of the transition. To illustrate this, consider again the 
net $\Psi$ from Example~\ref{ex:2}, and suppose that the input 
place $i$ contains tokens representing blocks of data as above. Moreover suppose
that the transition $s$ is just duplicating the blocks in order to apply $f$ and $g$ to each 
block. Now, the transition $j$ typically relies on joining the blocks of output data in
$l$ and $r$ with the same number. This is implemented by adding the condition
\begin{verbatim}
 ${l.num} :eq: ${r.num}
\end{verbatim}
to $j$. This modifies the behavior of the Petri net in a substantial way: The transition $j$ might stay 
disabled, even though there are enough tokens available on all input places. This change in behavior 
has quite some effects on the analysis of the net: For example, conflicts\footnote{A
\define{conflict} arises from a place $p$  holding at least one token if $p$ is an input place to more 
than one transition, but does not hold enough tokens to fire all these transitions.} might disappear,
while loop detection becomes harder. GPI-Space comes
with some analysis tools that take conditions into account. It is beyond the scope of this paper 
to go into detail on how to ensure correctness in the presence of conditions. Note, however, that 
the analysis is still possible in practically relevant situations, that is, in situations where a number 
of transitions formulates a complete and non-overlapping set of conditions (hence, there are no
conflicts or deadlocks).

\subsubsection{Example: Reduction and Parallel Reduction}\label{ex parallel red}

Parallelism can often be increased by splitting problems into smaller independent problems. This
requires that we combine (computer scientists say: reduce) the respective partial results into 
the final result.  Suppose, for example, that the partial results are obtained by executing a 
Petri net, say, $\Pi$, and that these results are attached as colours to tokens which are all added to 
the same place $p$ of $\Pi$. Further suppose that reducing the partial results means to apply an 
addition operator $+$. Then the reduction problem can be modeled by the Petri net

\vskip0.2cm
\noindent
\begin{petrinet}
 \node[place]      (p)              {$p$};
 \node[transition] (f) [right of=p] {$+$};
 \node[place]      (s) [right of=f] {$s$};
 \path[->]
       (p) edge (f)
       (f) edge [bend left] (s)
       (s) edge [bend left] (f)
 ;
\end{petrinet}

\vskip0.1cm
\noindent
which fits  into $\Pi$ locally as a subnet. The place $s$ holds the sum which is updated as long as
partial results are computed and assigned to the place $p$. The update operation executes $s_{i+1}=s_i+p_i$, where $s_i$ 
is the current value of the sum on $s$ and $p_i$ is one partial result on $p$. Note that this only makes 
sense if $+$ is  commutative and associative since the Petri net does not guarantee any order of execution. Then in the 
end, the value of the sum on $s$ is, say,  $s_0+p_0+\ldots+p_{n-1}$, where $s_0$ is the initial value of the state.
Note that $s_0$ needs to be set up by some mechanism not shown here. 

Often this is not what is wanted, for example because it may be hard to set up an initial state. The modified subnet 

\vskip0.2cm
\noindent
\begin{petrinet}
 \node[place]      (p)              {$p$};
 \node[transition] (f) [right of=p] {$+$};
 \node[place]      (s) [right of=f] {$s$};
 \node[transition] (z) [above of=f] {$\downarrow$};
 \node[place]      (a) [right of=z] {$\bullet$};
 \path[->]
       (p) edge (f)
       (a) edge (z)
       (p) edge [bend left] (z)
       (z) edge [bend left] (s)
       (f) edge [bend left] (s)
       (s) edge [bend left] (f)
 ;
\end{petrinet}

\vskip0.1cm
\noindent
computes $p_0+\ldots+p_{n-1}$ on $s$, and does not require any initial state. The first execution 
of this net fires the transition $\downarrow$ which just moves the single available token from $p$ to $s$, 
disabling itself. The transition $+$ is not enabled as long as $\downarrow$ has not yet fired, so there is no 
conflict between $\downarrow$ and $+$.

It is nice to see that Petri nets allow for local rewrites, local in the sense that no knowledge about the surrounding 
net is required in order to prove the correctness of the rewrite operation. Note, however, that both Petri nets above expose no 
parallelism: Whenever $+$ fires, the sum on $s$ is used, and no two incarnations of $+$ can run at the same time. 
The modified subnet 

\vskip0.2cm
\noindent
\begin{petrinet}
 \node[invisible]  (0)             {};
 \node[transition] (l) [above of=0] {$\to$};
 \node[transition] (r) [below of=0] {$\to$};
 \node[place]      (d) [left of=0,node distance=0.33\nd] {};
 \node[place]      (u) [right of=0,node distance=0.33\nd] {$\bullet$};
 \node[place]      (p) [left of=d] {$p$};
 \node[place]      (x) [right of=l] {$s$};
 \node[place]      (y) [right of=r] {$r$};
 \node[transition] (f) [right of=u] {$+$};
 \path[->]
       (d) edge [bend right] (r)
       (r) edge [bend right] (u)
       (u) edge [bend right] (l)
       (l) edge [bend right] (d)
       (p) edge [bend left] (l)
       (p) edge [bend right] (r)
       (l) edge (x)
       (r) edge (y)
       (x) edge [bend left] (f)
       (y) edge [bend right] (f)
 ;
 \draw[->]
       (f) to [out=-45,in=10] ($(r)+(0.66\nd,-0.66\nd)$) to [out=190,in=-90] (p)
 ;
\end{petrinet}

\vskip0.1cm
\noindent
shows a different behavior. Now, the tokens from $p$ are distributed on the two places
$s$ and $r$. As soon as both $s$ and $r$ hold a token, one incarnation
of $+$ can fire. At the same time, the transitions $\to$ can continue
to move tokens to $s$ and $r$, enabling $+$, and finally leading to
multiple incarnations of $+$ running at the same time. Note that the 
output of $+$ is fed back  to $p$, which makes much sense as it is just another partial result.

Altogether, this example shows how Petri nets can be used for a compact and executable specification 
of expected behavior, and then be changed gradually to get different non-functional properties.

\section{Modelling the Smoothness Test as a Petri Net}\label{sec Petri smooth}
Using the inherently parallel structure of the hybrid smoothness test within GPI-Space requires
a reformulation of our algorithms in the language of Petri nets. This will
also emphasize the possible concurrencies, which will automatically be exploited by GPI-Space.

The Petri net $\Gamma$ below

\vskip0.2cm
\noindent
\begin{petrinet}
  \node[place]      (i)              {$i$};
  \node[transition] (t) [right of=i] {$t$};
  \node[place]      (1) [right of=t] {};
  \node[transition] (r) [below of=1] {$r_t$};
  \node[invisible]  (0) [right of=1] {};
  \node[transition] (d) [above of=0] {$d$};
  \node[transition] (j) [below of=0] {$j$};
  \node[place]      (2) [right of=d] {};
  \node[place]      (3) [right of=j] {};
  \node[invisible]  (x) [right of=2] {};
  \node[transition] (s) [above of=x,node distance=0.5\nd] {$s$};
  \node[transition] (h) [below of=s] {$h_d$};
  \node[transition] (g) [below of=h] {$h_j$};
  \node[transition] (q) [below of=g] {$r_j$};
  \node[invisible]  (y) [below of=h,node distance=0.5\nd] {};
  \node[place]      (o) [right of=y] {$o$};
  \path[->]
        (i) edge (t)
        (t) edge (1)
        (1) edge (r)
        (1) edge (d)
        (1) edge (j)
        (d) edge (2)
        (j) edge (3)
        (2) edge (s)
        (2) edge (h)
        (3) edge (g)
        (3) edge (q)
        (h) edge (o)
        (g) edge (o)
        (s) edge [bend right=25] node [above] {} (i)
  ;
\end{petrinet}

\vskip0.1cm
\noindent
is a representation of the hybrid smoothness test as summarized in Algorithm~\ref{alg smoothnesstest}.
A computation starts with one token on the input place $i$, representing a triple $(I_W, I_X, q)$. At the top level, 
we will typically start with $I_W = \langle 0 \rangle$ and $q = 1$, that is, with $W=\mathbb{A}^n_{\bK}$.
Transition $t$ performs the check for (local) equality as in step 1 of Algorithm \ref{alg smoothnesstest}. Its output token represents, in addition to a copy of the input triple, a flag indicating the result of the check. By the use of conditions, it is ensured that the token will enable exactly one of the subsequent transitions. If the result of the check is \texttt{true}, which can only happen for tokens produced at the deepest level of recursion, 
then the variety is smooth in the current chart, and no further computation is required in this chart. In this case, the token will be removed by transition $r_t$. If the result is \texttt{false}, then the next action depends on whether the prescribed codimension limit $c$ in step 3 of Algorithm \ref{alg smoothnesstest} has been reached. 

If the codimension of $X$ in $W$ is $\leq c$, then transition $j$ will fire, which corresponds to executing Algorithm \ref{alg embJac} \texttt{EmbeddedJacobian}. 
If the Jacobian check gives \texttt{true}, the variety is smooth in the current chart, and the token will be removed by transition $r_j$.
 If the Jacobian check gives \texttt{false}, then the variety is not smooth. The transition $h_j$ will then add a token with the flag \texttt{false} 
to the output place $o$. Here, the letter $h$ stands for 'Heureka', the greek term for 'I have found'. If a Heureka occurs, all
remaining tokens except that one on $o$ are removed by clean-up transitions not shown in $\Gamma$, no new tasks are started, and 
all running work processes are terminated.\footnote{At current stage, the concept of a Heureka is not yet fully supported by both 
GPI-Space and {\sc{Singular}}, and is replaced in our implementation by a work-around.}

If the codimension of $X$ in $W$ is larger than $c$, then transition $d$ will fire, which corresponds to executing Algorithm \ref{alg deltacheck} 
\texttt{DeltaCheck} (considered as a black box at this point). The ensuing output token represents, in addition to a copy of $(I_W, I_X, q)$, a flag 
indicating the result of  \texttt{DeltaCheck}. If this result is \texttt{true}, then a descent to an ambient space of dimension one less is necessary. 
In this case, transition~$s$ fires, performing Algorithm \ref{alg descendembsmooth} \texttt{DescentEmbeddingSmooth}. This algorithm outputs 
a list of triples $(I_{W'},I_{X},q')$, each of which needs to be fed back to place $i$ for further processing. Note however, that in the formal 
description of Petri nets in Section \ref{sec petrinet}, we do not allow that a single firing of a transition adds more than one token to a 
single place. To model the situation described above in terms of a Petri net, we therefore introduce the subnet

\vskip0.2cm
\noindent
\begin{petrinet}
  \node[transition] (s)              {$s$};
  \node[place]      (1) [right of=s] {};
  \node[transition] (e) [right of=1] {$e$};
  \node[place]      (i) [right of=e] {$i$};
  \node[transition] (x) [below of=1] {$x$};
  \path[->]
        (s) edge (1)
        (1) edge [bend right] (e)
        (e) edge [bend right] (1)
        (e) edge (i)
        (1) edge (x)
  ;
\end{petrinet}

\vskip0.1cm
\noindent
between $s$ and $i$. Now, when firing, the transition $s$ produces a single output token, which represents a list $L$ of triples as above. 
As long as $L$ is non-empty, transition $e$ iteratively removes a single element from $L$ and assigns it to a token which is added to place 
$i$. Finally, transition $x$ deletes the empty list. 
These operations are formulated with expressions and conditions (see Subsection \ref{subsubsect:ext}), and can be parallelized as in Example \ref{ex parallel red}.
If, on the other hand,  \texttt{DeltaCheck} returns \texttt{false}, then  the variety is not smooth. Correspondingly, the transition $h_d$ fires, adding a token with 
the flag \texttt{false} to the output place $o$ and triggering a Heureka.

If all tokens within $\Gamma$ have been removed, all charts have been processed without detecting a singularity, and $X$ is smooth. In this case, a token 
with flag \texttt{true} will be added\footnote{This is done using some additional places and transitions not shown in $\Gamma$.} to the output place $o$.
Together with the fact that the recursion depth of Algorithm \texttt{DescentEmbeddingSmooth} is limited by the codimension of $X$ in $W$ and that each 
instance of it only produces finitely many new tokens, it is ensured that the execution of the Petri net terminates after a finite number of firings with exactly 
one token at $o$. GPI-Space automatically terminates if there are no more enabled transitions.

Note that, in addition to the task parallelism visible in $\Gamma$ (see also Example~\ref{ex:2}), all transitions in $\Gamma$ allow for multiple parallel 
instances, realizing data parallelism in the sense of Example~\ref{ex:1}.

So far, we have not yet explained how to model Algorithms \ref{alg descendembsmooth} to \ref{alg embJac}, 
on which Algorithm~\ref{alg smoothnesstest} is based. Algorithm \ref{alg embJac}\ \texttt{EmbeddedJacobian}, for instance, 
has been considered as a black box represented by transition $j$.
Note, however, that this algorithm exhibits a parallel structure of its own:
Apart from updating the ideal $Q$ in step 9 in order to use the condition $q \not\in Q$ as a termination criterion 
for the \textbf{while} loop in step 5,  the computations within the loop are independent from each other. 
Hence, waving step 9 and the check $q \not\in Q$ is a trivial way of
introducing data parallelism: Replace the subnet

\vskip0.2cm
\noindent
\begin{petrinet}
\node[place]      (i)              {};
  \node[transition] (j) [right of=i] {$j$};
  \node[place]      (1) [right of=j] {};
   \path[->]
    (i) edge (j)
    (j) edge node {} (1)
  ;
\end{petrinet}

\vskip0.1cm
\noindent
of $\Gamma$ by the Petri net

\vskip0.2cm
\noindent
\begin{petrinet}
  \node[place]      (i)              {};
  \node[transition] (p) [right of=i] {$p$};
  \node[place]      (1) [right of=p] {};
  \node[transition] (j) [right of=1] {$j'$};
  \node[place]      (2) [right of=j] {};
  \path[->]
    (i) edge (p)
    (p) edge node {} (1)
    (1) edge (j)
    (j) edge (2)
  ;
\end{petrinet}

\vskip0.1cm
\noindent
Here, the transition $p$ generates tokens corresponding to the submatrices $M$ of $\operatorname{Jac}(I_W)$ as described in step 
\ref{line Jac M} of the algorithm. 
Transition $j'$ performs the embedded Jacobian criterion computations in  steps \ref{line det jac} 
and \ref{line adj jac} to~\ref{line false jac}. 

Of course, in this version, the algorithm may waste valuable resources: 
There is a potentially large number of tokens generated by $p$ which lead to superfluous calculations further on.
This suggests to exploit the condition $q \not\in Q$ also in the parallel approach. That is, 
transition $j'$  should fire only until a covering of $X\cap D(q)$ has been obtained, and then trigger a Heureka 
for the \texttt{EmbeddedJacobian} subnet. However, at this writing, creating the infrastructure for a local Heureka 
is still subject to ongoing development. 
To remedy this situation at least partially, our current approach is to first compute all minors and collect them in $Q$,
and then to use a heuristic way\footnote{Radical membership seems to offer a more conceptual way: 
With notation as in Algorithm~\ref{alg embJac}, 
we have $g^m\in Q + I_W$ for some $m$. 
Given a representation 
$g^m = \sum_i a_i q_i + \sum_j b_j g_j$ with minors $q_i = \operatorname{det}(M_i)\in Q$, the $D(q_i)$  with $a_i\neq 0$ 
cover $X\cap D(g)$. However, finding such representations relies on
Gr\"obner bases computations and is, hence, not effective.} of iteratively dropping minors as long as $q \in Q$.
\begin{rem}
Both Algorithm \ref{alg deltacheck} \texttt{DeltaCheck} (see step \ref{line while delta}) and Algorithm~\ref{alg descendembsmooth} \texttt{DescentEmbeddingSmooth}  (see step \ref{line cover dec}) can be parallelized in a similar fashion. 
\end{rem}

The logic in all transitions is implemented in C++, using {\tt libSingular}, the C++-library version of {\sc Singular}, as the computational back-end. Some parts are 
written in the {\sc Singular} programming language, in particular those relying on functionality implemented in the {\sc Singular} libraries. In order to transfer the 
mathematical data from one work process to another one (possibly running on a different machine), the complex internal data structures need to be serialized. For 
this purpose, we use already existing functionality of {\sc Singular}, which relies on the so-called ssi-format. This serialization format has been created to efficiently 
represent {\sc Singular} data structures, in particular trees of pointers. The mathematical data objects communicated within the Petri net are stored in files located on 
a parallel file-system BeeGFS\footnote{See \url{https://de.wikipedia.org/wiki/BeeGFS}}, which is accessible from all nodes of the cluster. Alternatively, we could also use the virtual memory layer provided by 
GPI-Space. However, on the cluster used for our timings, the speed of the (de)serialization is limited by the CPU and not the underlying storage medium. 

The implementation of the hybrid smoothness test can be used through a startup binary, which is suitable for queuing systems  commonly used in computer clusters. 
Moreover, there is also an implementation of a dynamical module for {\sc Singular}, which allows the user to directly run the implementation from within the {\sc Singular} 
user interface. It should be noted that neither the instance of {\sc Singular} providing the user interface nor {\tt libSingular} had to be modified in order to cooperate with GPI-Space.

\section{Applications in Algebraic Geometry and Behaviour of the Smoothness Test}\label{sec timings}

To demonstrate the potential of the hybrid smoothness test and its implementation as described 
in Section \ref{sec Petri smooth}, we apply it  to problems originating from current research in 
algebraic geometry. We focus on  two classes of surfaces of general type, which provide good test examples 
since their defining ideals are quite typical for those arising in advanced constructions in algebraic geometry: 
They have large codimension, and  their rings of polynomial functions are Cohen-Macaulay and even Gorenstein. 
Due to their structural properties,  rings of these types are of fundamental importance in algebraic geometry. 

We begin by giving some background on our test examples, and then provide timings and investigate how the 
implementation scales with the number of cores.

\subsection{Applications in Algebraic Geometry}\label{sec applications}

The concept  of moduli spaces provides geometric solutions to classification problems
and is ubiquitous in algebraic geometry where we wish to classify algebraic varieties
with prescribed invariants. There is a multitude of abstract techniques for
the qualitative and quantitative study of these spaces, without, in the first instance,
taking explicit equations of the varieties under consideration into account. Relying on 
equations and their syzygies (the relations between the equations), on the other hand, we may manipulate geometric objects 
using a computer. In particular, if an explicit way of constructing a general element of a 
moduli space $M$ is known to us, we may  detect geometric properties of $M$ by studying 
such an element computationally. Deriving a construction is the innovative and often 
theoretically involved part of this approach, while the technically difficult part, the 
verification of the properties of the constructed objects, is left to the machine. 

Arguably, the most important property to be tested here is smoothness. To provide a basic example 
of how smoothness affects the properties of the constructed variety, note that a smooth 
plane cubic is an elliptic curve (that is, it has geometric genus one), whereas 
a singular plane cubic is a rational curve (which has geometric genus zero).

The study of (irreducible smooth projective complex)
surfaces with geome\-tric genus and irregularity $p_g=q=0$ has a rich history,
and is of importance for several reasons, with surfaces of general type providing
particular challenges (see \cite {BHPV}, \cite{zbMATH05974000}). The self-intersection of a
canonical divisor $K$ on a minimal surface of general type with
$p_g=q=0$ satisfies $1\leq K^2\leq 9$, where the upper bound is given by the
Bogomolov-Miyaoka-Yau inequality  (see \cite[VII,
4]{BHPV}).  Hence, these surfaces belong to only finitely many components 
of the Gieseker moduli space for surfaces of general type \cite{MR0498596}.
Interestingly enough, Mumford asked whether their classification can be done by a computer.

Of particular interest among these surfaces are the \emph{numerical Godeaux} and \emph{numerical Campedelli
surfaces}, which satisfy $K^2=1$ and $K^2=2$, respectively. As Miles Reid  puts it \cite{ReidOV}, these ``are in some sense the first
cases of the geography of surfaces of general type, and it is somewhat embarrassing that
we are still quite far from having a complete treatment of them''. Their study is
``a test case for the study of all surfaces of general type''.

For our timings, we focus on two specific examples, each defined over a finite prime field $\bk$.
Though, mathematically, we are interested in the geometry of the surfaces in characteristic zero,
computations in characteristic $p$ (which are less expensive) are enough to demonstrate 
the behavior of the smoothness test.

The first example is a numerical Campedelli surface $X$ with torsion group  $\mathbb{Z}/6 \mathbb{Z}$, which has been constructed 
in \cite{NP1} (we work over the finite field $\bk=\mathbb{Z}/103\mathbb{Z}$ which contains, as required by the construction, a primitive 3rd root of 
unity). The construction yields $X$ as a $\mathbb{Z}/6 \mathbb{Z}$-quotient of a covering surface $\widetilde{X}$ which, in turn, is realized as a 
subvariety of the weighted projective space $\mathbb{P}_\bK(1,1,1,1,1,2,2,2)$, where $\bK=\overline{\bk}$. That is, the homogeneous coordinate ring of the ambient space
is a polynomial ring with $5$ variables of degree 1 and $3$ variables of degree 2, and the codimension of $X$ in that space is 5.
In fact, $\widetilde{X}$ is constructed from a hypersurface in projective 3-space $\mathbb{P}_\bK^{3}$ using Kustin-Miller unprojection \cite{KM}.
This iterative process  increases in every iteration step the codimension of a given Gorenstein ring  
by one, while retaining the Gorenstein property. See \cite{BP3,BPlib} for an outline and implementation.
We use the hybrid smoothness test to verify the quasi-smoothness of $\widetilde{X}$, that is, the smoothness of 
the affine cone over $\widetilde{X}$ outside the origin. This amounts to apply the test
in each of the 8 (affine) coordinate charts of $\mathbb A^8_\bK\setminus \{0\}$. Note that in general,
quasi-smoothness does not automatically guarantee smoothness due to the singularities of
the weighted projective space. In our case, however, the smoothness of both surfaces 
$\widetilde{X}$ and $X$ follows from the quasi-smoothness of $\widetilde{X}$  by a straightforward 
theoretical argument.

The second example is a  numerical Godeaux surface with trivial torsion group. It is taken from
ongoing research work by Isabel Stenger, who uses a construction method suggested by
Frank-Olaf Schreyer in \cite{FOS-OW}. The resulting surface is a subvariety of $\mathbb{P}_\bK^{13}$ (of 
codimension $11$) which is cut out  by $38$ quadrics (and is again realized over the finite field $\mathbb{Z}/103 \mathbb{Z}$). 
Using our implementation, we verify the smoothness of the surface by verifying smoothness in each of the 14 
coordinate charts of $\mathbb{P}_\bK^{13}$. Note that to the best of our knowledge,
this cannot be done by other means.

\subsection{Behavior of the Smoothness Test}
The timings in this subsection are taken on a cluster provided by Fraunhofer ITWM Kaiserslautern.  This cluster consists of $192$ nodes, each of which
has $16$ Intel Xeon E5-2670 cores running at $2.6$ GHz with $64$ GB of RAM (so the cluster has a total of $3072$ cores and $12$ 
TB of RAM). The nodes are connected via FDR Infiniband. Note that the cores are utilized by GPI-Space in a non-hyperthreading way, that is, with a maximum of $16$ jobs per node.

In the case of the the numerical Campe\-del\-li surface, we apply the hybrid smoothness test  
with a descent in codimension to minors of size $2\times 2$. Timings are given in Table \ref{tab timings}
\begin{table}[h]
\caption{Run-times of the hybrid smoothness test when applied to the numerical Campedelli surface\medskip}%
\label{tab timings}%
\begin{tabular}
[c]{c|ccc|c}%
number of cores & time / sec &  & number of cores & time / sec\\\cline{1-2}%
\cline{4-5}%
\textbf{1} & \textbf{2\,686.98} &  & 48 & 68.64\\
\textbf{2} & \textbf{1350.67} &  & \textbf{64} & \textbf{51.98}\\
\textbf{4} & \textbf{684.77} &  & 80 & 39.64\\
6 & 466.96 &  & 96 & 32.30\\
\textbf{8} & \textbf{356.18} &  & 112 & 27.56\\
10 & 290.75 &  & \textbf{128} & \textbf{26.15}\\
12 & 245.19 &  & 160 & 21.36\\
14 & 215.46 &  & 192 & 19.10\\
\textbf{16} & \textbf{191.65} &  & 224 & 18.52\\
\textbf{32} & \textbf{99.06} &  & \textbf{256} & \textbf{18.41}%
\end{tabular}
\end{table}
for $1$ up to $256$ cores (the powers of two are shown in bold), where we always take the average over $100$ runs. See also Figure \ref{fig S5 timings}  for 
a visualization, where the data points correspond to the entries of Table \ref{tab timings}, and the plotted curve is a 
least-square fit of the run-times 
using a hyperbola.
\begin{figure}
\caption{Display of  the run-times from Table \ref{tab timings}  for the numerical Campedelli surface (in seconds)\medskip}%
\label{fig S5 timings}
\begin{center}
\includegraphics[
height=12cm,
]%
{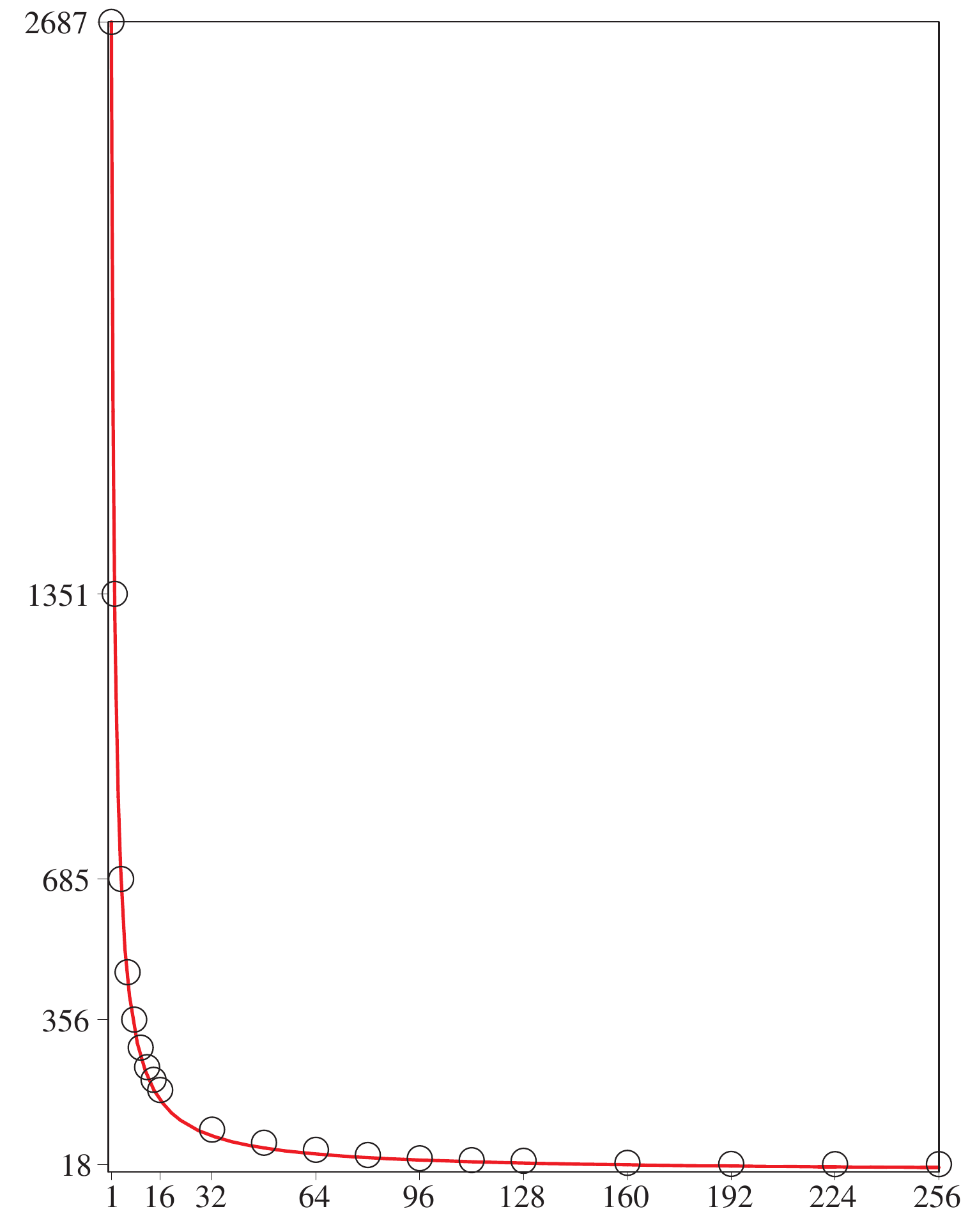}%
\end{center}
\end{figure}
In Figure \ref{fig scaling}, 
\begin{figure}
\caption{Scaling with the number of cores of the run-times from Table \ref{tab timings} for the numerical Campedelli surface  \medskip}%
\label{fig scaling}
\smallskip
\begin{center}
\includegraphics[
height=5.05cm,
]%
{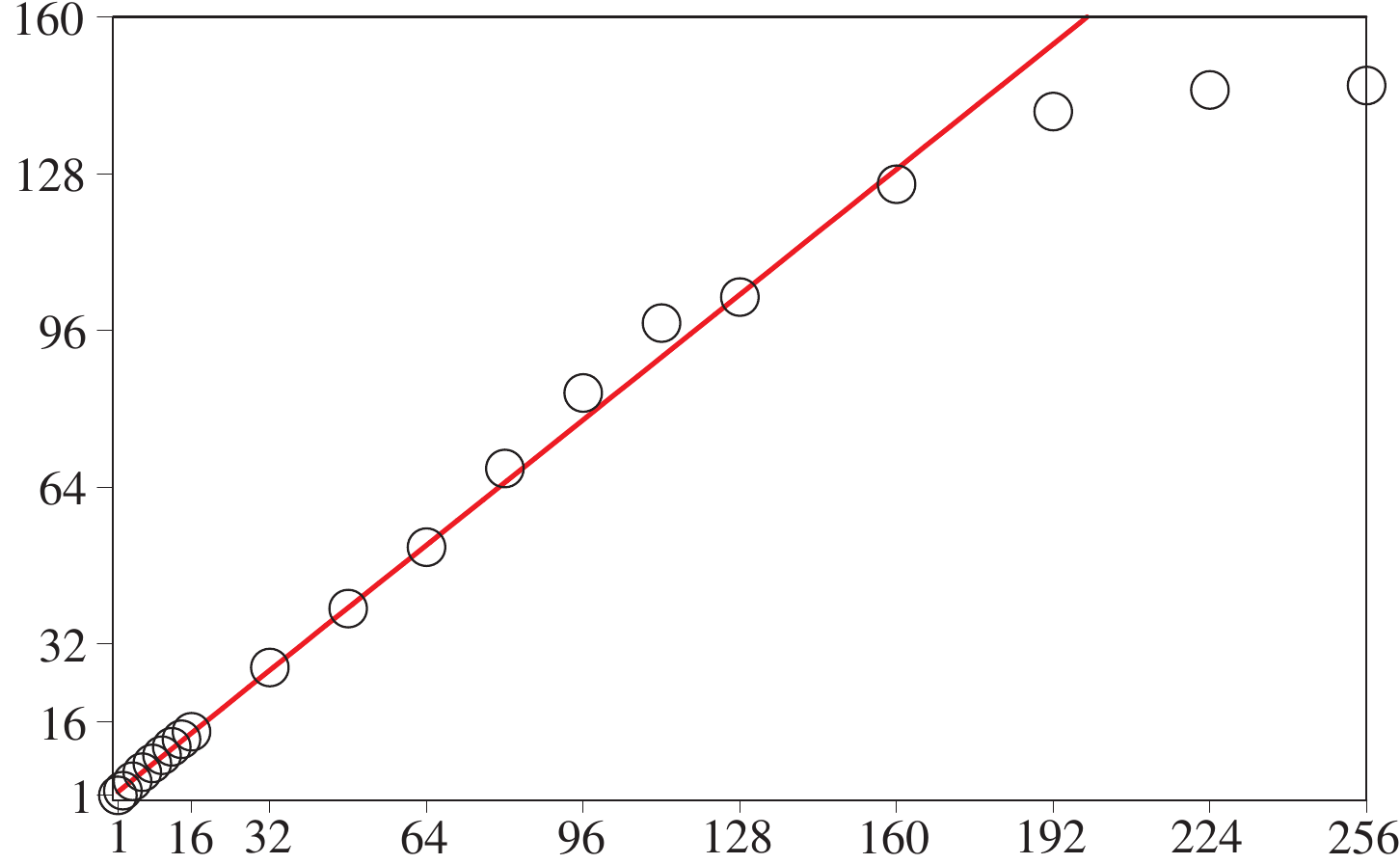}%
\end{center}
\end{figure}
we show how the implementation scales with the number of cores by plotting the speedup-factor (relative to the single core run-time) versus the number of cores. We
observe a linear speedup up to $160$ cores. Figure \ref{fig efficiency} visualizes the parallel efficiency (speedup divided by number of cores) of the computation.
\begin{figure}
\caption{Parallel efficiency determined from run-times in Table \ref{tab timings} for the numerical Campedelli surface  \medskip}%
\label{fig efficiency}
\smallskip
\begin{center}
\includegraphics[ height=5.32cm
]%
{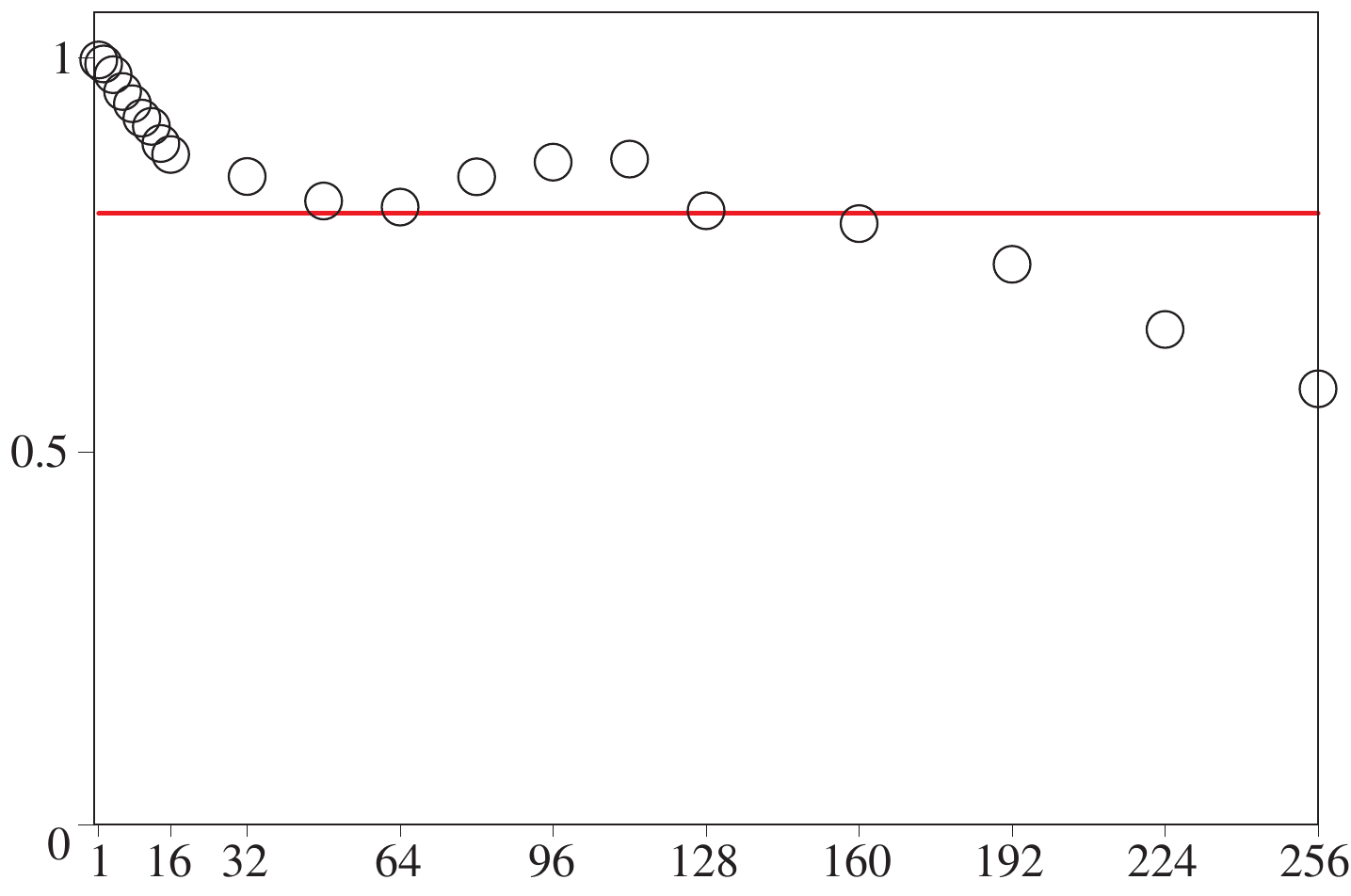}%
\end{center}
\end{figure}

To give some explanation for this observation, we note that starting from the 8 affine coordinate 
charts of $\mathbb A^8_\bK\setminus \{0\}$, the hybrid smoothness test in its current implementation 
may branch into up to 323 charts at the leaves of the resulting tree of charts. 
As it turns out, however, already a proper subset of the  coordinate charts is enough to
cover the affine cone over $\widetilde{X}$ outside the origin, and the algorithm will terminate 
once this situation has been achieved. Typically, the algorithm finishes with a total of about  $240$ charts. 
Hence, we cannot expect any scaling beyond this number of cores. Note that the descent in codimension involves a smaller number of charts, which also limits the scaling. Applying the projective Jacobian criterion, that is, 
computing the ideal $J$ generated by the codimension-sized minors of the Jacobian matrix and saturating the ideal $I_X+J$
with respect to the irrelevant maximal ideal (which is generated by all variables), takes about $580$ 
seconds on one core and uses about $15$ GB of memory. We observe that, while single runs of the 
massively parallel implementation take more than these $580$ seconds, by passing 
to a larger number of cores, we can achieve a speedup of at least factor $30$ compared to the projective
Jacobian criterion. We also remark that while the computation of the minors in the Jacobian criterion 
can be done in parallel, the subsequent saturation (which takes most of the total computation time) is 
an inherently sequential process. With regard to memory usage, each of the individual Jacobian criterion 
computations in Algorithm \ref{alg embJac} \texttt{EmbeddedJacobian}  does not exceed $450$ MB of RAM  
(due to the small size of minors after the descent).

In case of the numerical Godeaux surface, we apply the hybrid smoothness test  with a descent in codimension down
to minors of size $3\times 3$. So far, smoothness of this surface could not be verified by the projective Jacobian criterion, 
which runs out of memory exceeding the available $384$ GB of RAM of the machine we used. The hybrid 
smoothness test easily handles this example, using a maximum of $3.1$ GB of RAM for one of the individual Jacobian 
criterion computations after the descent. Timings are given in Table \ref{tab timings2} 
\begin{table}[h]
\caption{Run-times of the hybrid smoothness test when applied to the numerical Godeaux surface\medskip}%
\label{tab timings2}%
\begin{tabular}
[c]{c|c|c}%
number of nodes & number of cores & time / sec\\\hline\cline{1-2}%
\textbf{1} &\textbf{16} & \textbf{53\thinspace000}\\
\textbf{2} &\textbf{32} & \textbf{33\thinspace000}\\
\textbf{4} &\textbf{64} & \textbf{12\thinspace200}\\
\textbf{8} &\textbf{128} & \textbf{3\thinspace100}\\
\textbf{16} &\textbf{256} & \textbf{2\thinspace460}%
\end{tabular}
\end{table}
for $16$ up to $256$ cores, where we always take the average over $10$ runs.
See also Figure \ref{fig godeaux time} for a visualization, where the data points correspond to the entries   of Table \ref{tab timings2} 
and the plotted curve is again a least-square fit of the run-times using a hyperbola. 
\begin{figure}
\caption{Display of  the run-times from Table \ref{tab timings2} for the numerical Godeaux surface   (in seconds) \medskip}%
\label{fig godeaux time}
\smallskip
\includegraphics[
height=12cm
]%
{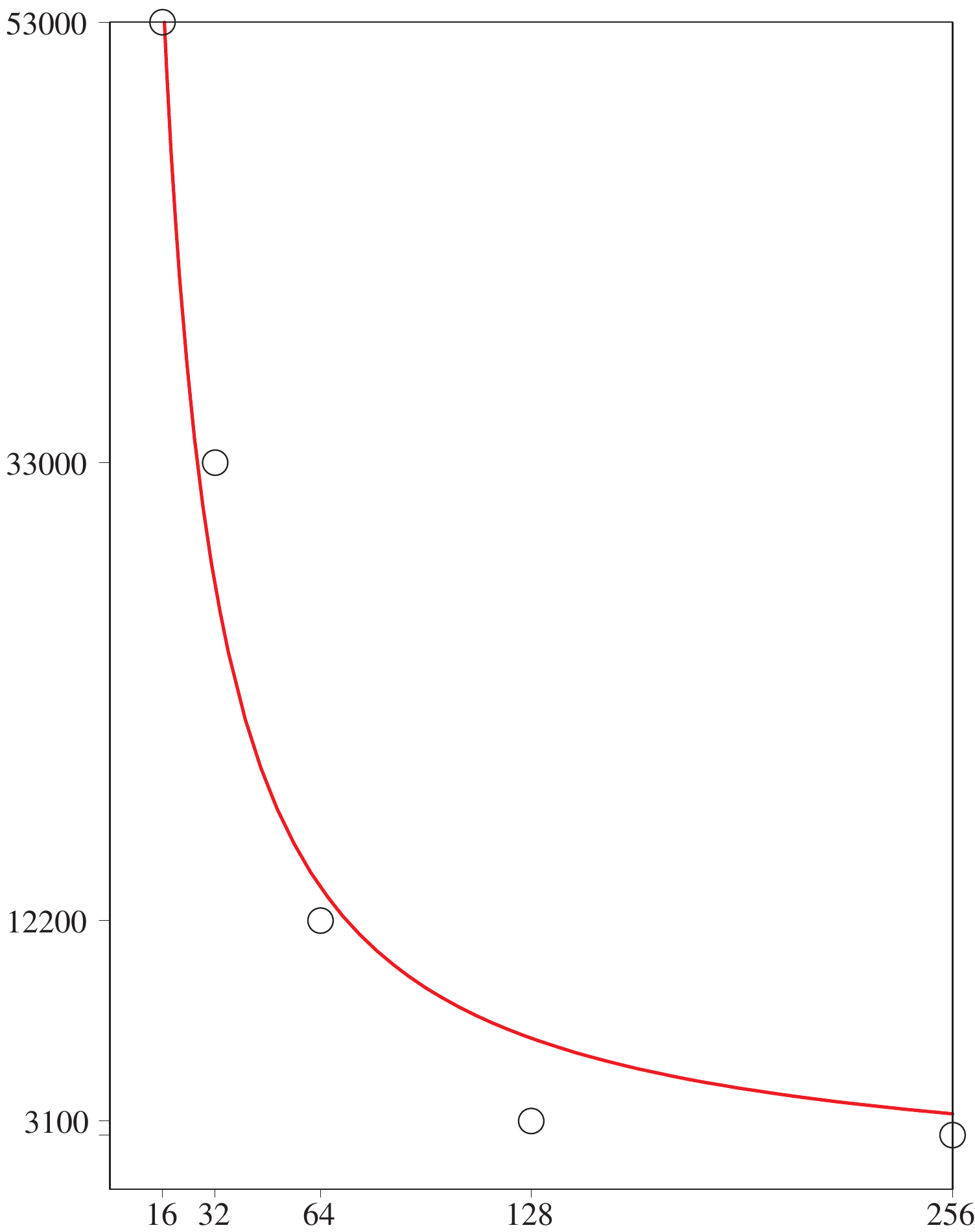}%
\end{figure}

We observe that in this example, we actually get a super-linear speedup, that is, when doubling the number of cores used by the algorithm, 
the computation time drops by more than a factor of two. We have identified two reasons for this effect. 

One  reason is purely technical:  If more cores than tasks are available to the algorithm, that is, the load factor is smaller than one, 
then each individual computation can use a larger memory bandwidth, which speeds up the computation. To indicate the
impact of the workload on the performance, Figure~\ref{fig speicher} 
\begin{figure}
\caption{Run-times for parallel individual Jacobian criterion computations on one node (in seconds)\medskip}%
\label{fig speicher}
\smallskip
\includegraphics[
height=12cm
]%
{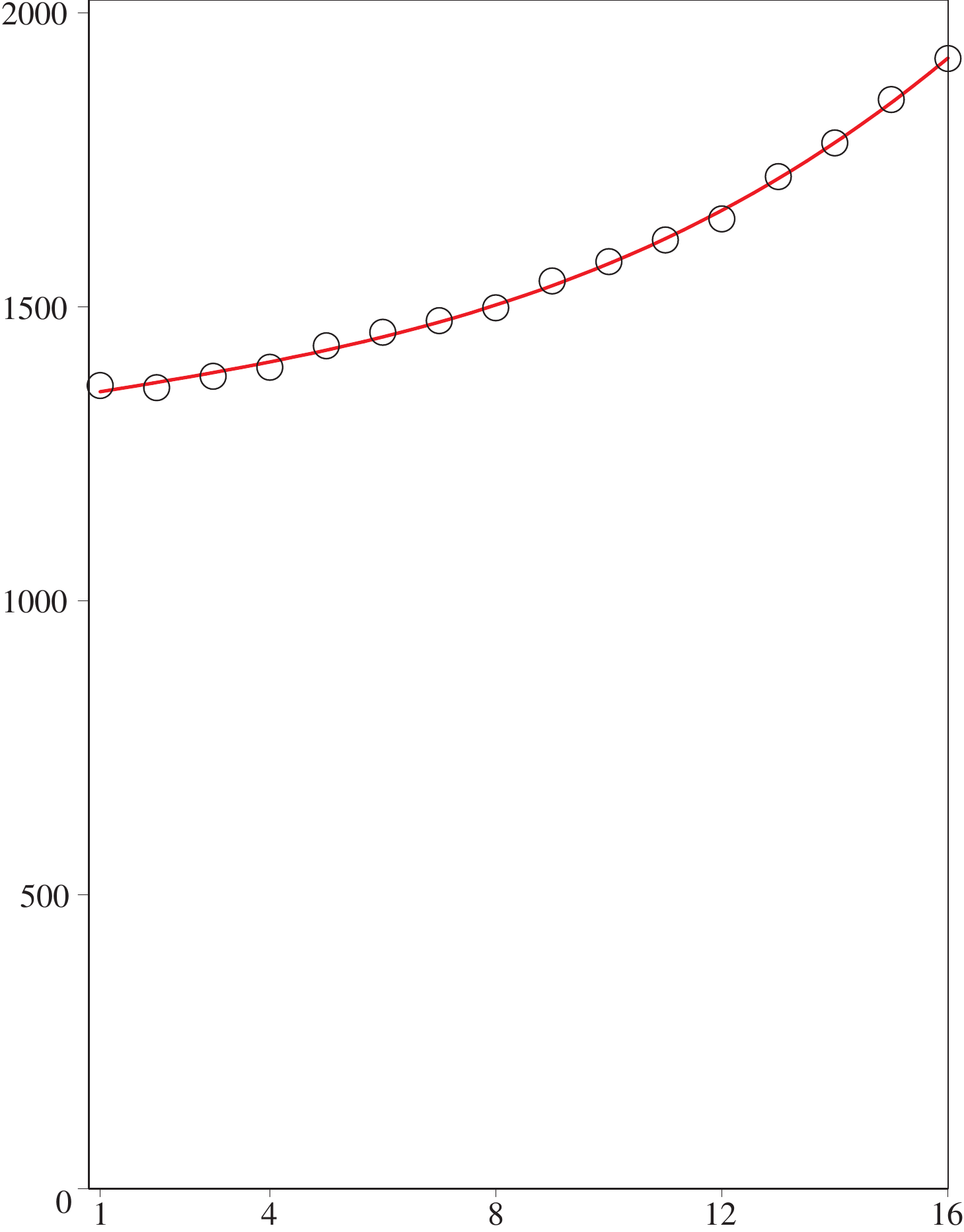}%
\end{figure}
shows the time used for parallel runs of a given number of copies of a single Jacobian criterion computation 
on a single node. While the load factor of the smoothness test  is close to $1$ when executed on less 
than $64$ cores, it drops to about $0.7$ on $256$ cores, which amounts to a speedup of about $30\%$.  

More important 
is the second reason,
which stems from the structure of the algorithm and the mathematics behind the surface under consideration: 
The smoothness of this surface is determined by considering (on the first level of the algorithm) all $14$ affine 
charts of the ambient projective space $\mathbb P_\bK^{13}$. The algorithmic subtrees of $4$ of these charts do not terminate during the descent in 
codimension within $50\,000$ seconds, while the final covering obtained by the algorithm will always consist of 
the same $4$ of the remaining $10$ charts: Since the implementation branches into all available 
choices in a massively parallel way and terminates once the surface is completely covered by charts, it will 
automatically determine that choice of charts which leads to the smoothness certificate in the fastest possible 
way. Note that the $10$ remaining charts above involve a total of 
$115$ sub-charts, so we cannot expect much scaling beyond this number of cores.

We have done a simulation of this behavior of the Petri net using the actual computation times of the individual sub-steps of the algorithm for all available choices (sampling all timings for the sub-steps in the same environment). The simulated scaling with the number of cores matches very well the actual behavior of the implementation on the cluster, see Figure  \ref{fig simulation}
\begin{figure}
\caption{Simulated timings for determining smoothness of the numerical Godeaux surface via the hybrid smoothness test\medskip}%
\label{fig simulation}
\smallskip
\includegraphics[
height=13cm
]%
{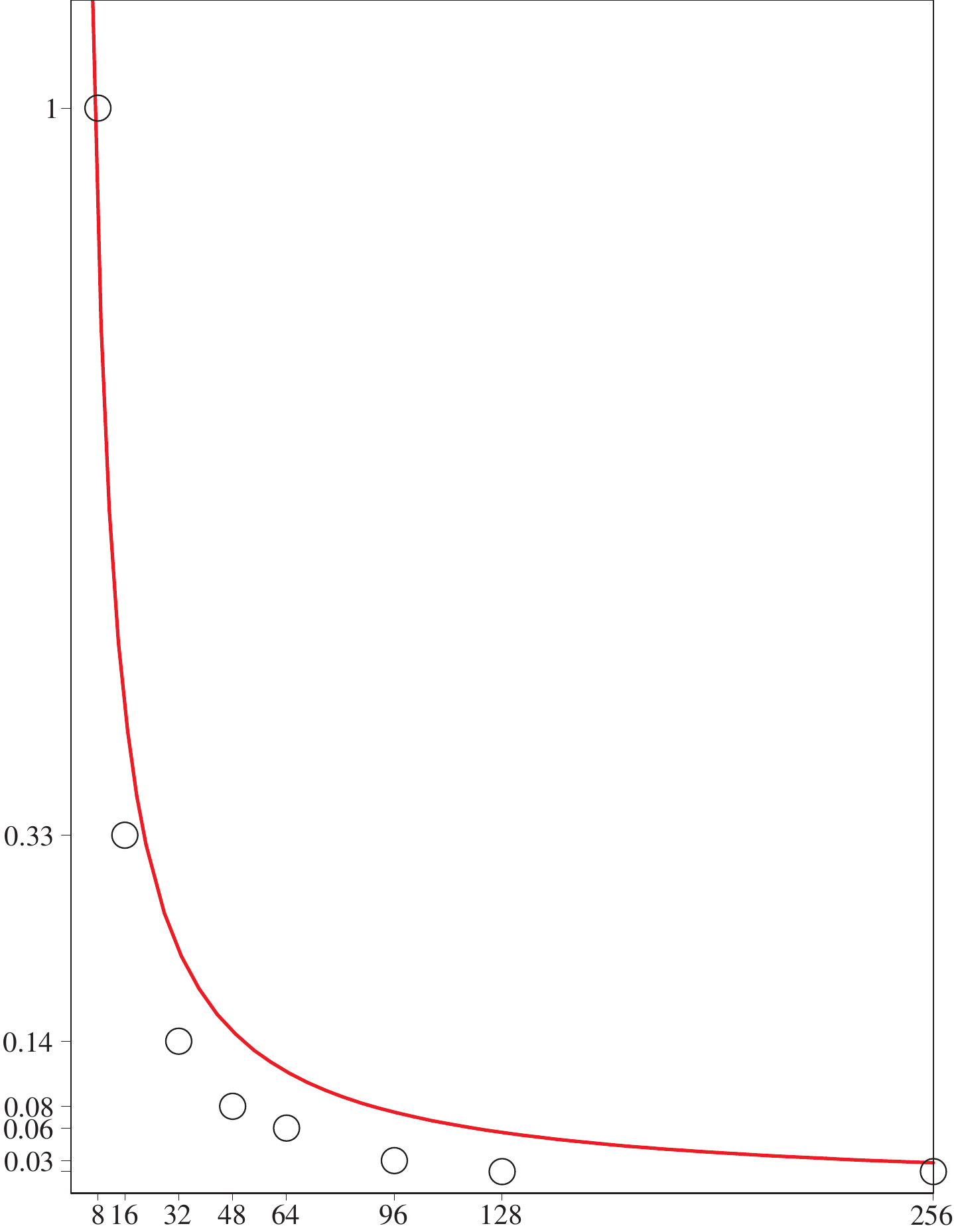}%
\end{figure} for the synthetic timings (normalized to value one  for $8$ cores): With up to $4$ available cores, all cores will run into an unfavorable chart with probability almost $1$, while for $8$ to $128$ cores, we observe a  super-linear speedup. As expected from the geometric structure of the specific problem, the simulation does not show a significant further speedup beyond $128$ cores. 

To summarize, when working with charts, we have the flexibility of  choosing a covering which leads to fast individual 
computations that are well-balanced with regard to their run-time, resulting in a good performance of the overall parallel algorithm. 
Due to the unpredictability of the individual computations, this choice cannot be made a priori in a heuristic way. However, with 
a massively parallel approach, the best possible choice is found automatically by the algorithm. The chart based nature 
of the smoothness test reflects a fundamental paradigm of algebraic geometry, the description of schemes and sheaves 
in terms of charts. One can, hence, expect that a similar approach will also be useful for further applications in algebraic 
geometry, for example, in the closely related problem of resolution of singularities.

\vspace{0.1in} \noindent\emph{Acknowledgements}. We would like to thank Bernd L\"orwald,
 Stavros Papadakis, Gerhard Pfister, Christian Reinbold, Bernd Schober, Hans Sch\"onemann, and Isabel Stenger for helpful discussions.

\end{document}